\documentclass[11pt, reqno]{amsart}
\usepackage{amsmath, amsfonts, amsthm, amssymb, multicol, mathtools, dsfont, mathrsfs}
\usepackage{graphicx}
\usepackage{float, hyperref}
\usepackage{scalerel,stackengine, subcaption}
\usepackage[usenames,dvipsnames,x11names]{xcolor}
\usepackage{enumitem}
\usepackage{pgfplots}
\usepgfplotslibrary{fillbetween}
\pgfplotsset{width=10cm,compat=1.9}
\usepackage[square,sort,comma,numbers]{natbib}
\allowdisplaybreaks

\makeatletter
\g@addto@macro\bfseries{\boldmath}
\makeatother

\makeatletter
\def\@setauthors{%
  \begingroup
  \def\thanks{\protect\thanks@warning}%
  \trivlist
  \centering\footnotesize \@topsep30\p@\relax
  \advance\@topsep by -\baselineskip
  \item\relax
  \author@andify\authors
  \def\\{\protect\linebreak}

  \normalsize\lowercase{\authors}%
  
	\ifx\@empty\contribs
  \else
    ,\penalty-3 \space \@setcontribs
    \@closetoccontribs
  \fi
  \endtrivlist
  \endgroup
}
\def\@settitle{\begin{center}
\LARGE\lowercase{\@title}
  \end{center}%
}
\makeatother
\newcommand{\authoremail}[1]{\email{\href{mailto:#1}{\color{lightblue}{#1}}}}
\newcommand{\authoraddress}[1]{\address{\normalfont{#1}}}

\numberwithin{equation}{section}

\newtheorem{thm}{Theorem}[section]
\newtheorem{lma}[thm]{Lemma}
\newtheorem{cor}[thm]{Corollary}

\newtheorem{prop}[thm]{Proposition}

\newtheorem{conj}[thm]{Conjecture}

\renewcommand{\epsilon}{\varepsilon}
\newcommand{\eps}{\varepsilon}

\newcommand{\rd}{\mathbb{R}^d}

\renewcommand{\geq}{\geqslant}
\renewcommand{\leq}{\leqslant}

\newcommand{\frd}{\dim_{\textup{Fr}}}

\newcommand{\fs}{\dim^\theta_{\mathrm{F}}}
\newcommand{\fd}{\dim_{\mathrm{F}}}
\newcommand{\sd}{\dim_{\mathrm{S}}}

\newcommand{\R}{\mathbb{R}}

\newcommand{\C}{\mathbb{C}}

\newcommand{\N}{\mathbb{N}}
\newcommand{\spt}{\text{spt}\,}

\newcommand{\J}{\mathcal{J}}

\makeatletter
\DeclareRobustCommand\widecheck[1]{{\mathpalette\@widecheck{#1}}}
\def\@widecheck#1#2{%
    \setbox\z@\hbox{\m@th$#1#2$}%
    \setbox\tw@\hbox{\m@th$#1%
       \widehat{%
          \vrule\@width\z@\@height\ht\z@
          \vrule\@height\z@\@width\wd\z@}$}%
    \dp\tw@-\ht\z@
    \@tempdima\ht\z@ \advance\@tempdima2\ht\tw@ \divide\@tempdima\thr@@
    \setbox\tw@\hbox{%
       \raise\@tempdima\hbox{\scalebox{1}[-1]{\lower\@tempdima\box
\tw@}}}%
    {\ooalign{\box\tw@ \cr \box\z@}}}
\makeatother

\stackMath
\newcommand\reallywidehat[1]{%
\savestack{\tmpbox}{\stretchto{%
  \scaleto{%
    \scalerel*[\widthof{\ensuremath{#1}}]{\kern.1pt\mathchar"0362\kern.1pt}%
    {\rule{0ex}{\textheight}}
  }{\textheight}%
}{2.4ex}}%
\stackon[-6.9pt]{#1}{\tmpbox}%
}
\definecolor{lightblue}{HTML}{2B77A4}
\colorlet{plotblue}{LightSkyBlue3!80}
\definecolor{darkred}{HTML}{9E0D0D}
\definecolor{purp}{HTML}{d603a9}
\definecolor{dartmouthgreen}{HTML}{00A64F}
\definecolor{Junglegreen}{HTML}{00A99A}
\definecolor{yellowcolour}{HTML}{f07c02}
\hypersetup{
	colorlinks=true,
	linkcolor=darkred,
	urlcolor=darkred,
	citecolor=lightblue
}
\title{$L^2$ restriction estimates from the Fourier spectrum}

\usepackage{fancyhdr}
\author{Marc Carnovale}
\authoraddress{Marc Carnovale}
\authoremail{carnom2@nationwide.com}

\author{Jonathan M. Fraser}
\authoraddress{Jonathan M. Fraser, University of St Andrews, Scotland}
\authoremail{jmf32@st-andrews.ac.uk}
\thanks{JMF was  financially supported by a  \emph{Leverhulme Trust Research Project Grant} (RPG-2023-281)  and an \emph{EPSRC Standard Grant} (EP/Y029550/1).}

\author{Ana E. de Orellana}
\authoraddress{A. E. de Orellana, University of St Andrews, Scotland}
\authoremail{aedo1@st-andrews.ac.uk}
\thanks{AEdO was financially supported by the University of St Andrews.}

\date{}

\begin{document}
\thispagestyle{empty}

\begin{abstract}
The Stein--Tomas  restriction theorem is an important  result in Fourier restriction theory.  It gives a range of $q$  for which $L^q\to L^2$ restriction estimates hold for a given measure, in terms of the Fourier and Frostman dimensions of the measure. We generalise this result by using the Fourier spectrum; a family of dimensions that interpolate between the Fourier and Sobolev dimensions for measures.  This gives us a continuum of Stein--Tomas type estimates, and optimising over this continuum gives a new $L^q\to L^2$ restriction theorem which often outperforms the Stein--Tomas result.  We also provide results in the other direction by giving a range of $q$ in terms of the Fourier spectrum for which $L^q\to L^2$ restriction estimates fail, generalising an observation  of Hambrook and {\L}aba.  We illustrate our results with several examples, including the surface measure on the cone, the moment curve, and several fractal measures. 
\end{abstract}
\maketitle
\tableofcontents

\section{Introduction}

\subsection{The restriction problem} \label{sec:restrictionIntro}

Given a non-zero, finite, compactly supported, Borel measure $\mu$ on $\rd$, the celebrated \emph{restriction problem} asks when is it meaningful to restrict the Fourier transform of a function to the support of $\mu$.  More precisely,  for which $p,q \in [1,\infty]$  does it hold that
\begin{equation}\label{eq:restriction}
    \|\widehat{f}\|_{L^{p'}(\mu)} \lesssim \|f\|_{L^{q'}(\rd)},
\end{equation}
for a uniform constant $C \geq 1$ for all $f\in L^{q'}(\rd)$.  Here and throughout the paper we write $A\lesssim B$ or $A\gtrsim B$ if there exists a uniform constant $C>0$ such that $A\leq CB$ or $A\geq CB$ respectively, and $A\approx B$ if both $A\lesssim B$ and $A\gtrsim B$ hold. If we wish to emphasise that $C$ depends on some parameter $\lambda$, it will be denoted with a subscript as $\lesssim_{\lambda}$, $\gtrsim_{\lambda}$, or $\approx_{\lambda}$. Also, for   $p,q\in[1,\infty]$, we write $p'$ and $q'$ to refer to their conjugate exponents, i.e. they will satisfy $\frac{1}{p}+ \frac{1}{p'} = \frac{1}{q} + \frac{1}{q'} = 1$.

By duality of $L^p$ spaces, \eqref{eq:restriction} is equivalent to
\begin{equation*}
    \|\widehat{f\mu}\|_{L^{q}(\rd)}\lesssim \|f\|_{L^p(\mu)}
\end{equation*}
with \eqref{eq:restriction} referred to as an $L^{q'} \to L^{p'}$ restriction estimate and the dual form as  an $L^{p} \to L^{q}$ extension estimate. What is more, if $p = 2$, both restriction and extension estimates are equivalent to
\begin{equation}\label{eq:extension}
    \|\widehat{\mu}*f\|_{L^q(\rd)}\lesssim \|f\|_{L^{q'}(\rd)}.
\end{equation}

We suppose from now on that the support of $\mu$ is a Lebesgue null set in $\rd$ (and therefore we are genuinely attempting to restrict).  If  $q'=1$ then $f$ is integrable and  $\widehat{f}$ is continuous and bounded, and so estimates of the form \eqref{eq:restriction} will be possible for all $p'$. However, since the Fourier transform is an $L^2(\rd)$ isometry, if $ q'= 2$, $\widehat{f}\in L^2(\rd)$ and is only defined Lebesgue almost everywhere and so  \eqref{eq:restriction} will not hold for any $p'$. Thus, the question is only  interesting for the range $1< q'<2$.

There are many important open problems in restriction theory.  For example, for the surface measure on both the sphere and paraboloid, the conjecture is that \eqref{eq:restriction} holds if
\begin{equation} \label{sphereconjecture}
    \frac{d-1}{p'}\geq \frac{d+1}{q}\quad \text{and}\quad q>\frac{2d}{d-1},
\end{equation}
and is still open for $d\geq 3$. We refer the reader to \cite{Mat15}, and \cite{Dem20} for a more thorough presentation of its history. The interest on restriction estimates for (surface measures on) manifolds is partially motivated by its connection to PDEs (see \cite{OeS24} for other applications). For example, restriction estimates for the cone (see Section~\ref{sec:cone}) and the parabola, lead to $L^p$ estimates for the solutions of the wave and Schr\"odinger equations, respectively. It also has deep connections to harmonic analysis and geometric measure theory.  For example, the restriction conjecture for the sphere implies the Kakeya maximal function conjecture, which in turn implies the Kakeya set conjecture.

\subsection{Stein--Tomas restriction}

Stein observed that non-trivial estimates are often possible when $\mu$ is the surface measure on a smooth curved manifold, but that nothing can be said when the surface has flat pieces.  In general, such curvature features can be captured by Fourier decay and this led to, over many years, the Stein--Tomas restriction theorem.  The most general version of this is due to  Bak--Seeger \cite{BS11}, but builds on work of Stein (see \cite{Ste93}), Tomas \cite{Tom75}, Mochenhaupt \cite{Moc00}, Mitsis \cite{Mit02}, and others (see also \cite{Fef70,CS72}).  This theorem considers the  case $p=2$ and gives a non-trivial range of $q$ for which \eqref{eq:restriction} holds in terms of the Frostman exponent and Fourier decay of $\mu$.  In particular, if $\alpha>0$ and $\beta>0$ are such that
\begin{equation}\label{eq:FrostmanExp}
  \mu(B(x,r)) \lesssim r^\alpha
\end{equation}
for all $x \in \rd$ and $r>0$, and 
\begin{equation}\label{eq:FourierDecay}
  \big|\widehat{\mu}(\xi) \big|^2 \lesssim |\xi|^{-\beta} 
\end{equation}
for all $\xi \in \rd$, then 
for all $f\in L^2(\mu)$, the estimate
\begin{equation}\label{eq:restrictionp2}
    \|\widehat{f\mu}\|_{L^q(\rd)}\lesssim \|f\|_{L^2(\mu)}
\end{equation}
holds for all
\begin{equation}\label{eq:SteinTomas}
q \geq 2+4 \frac{d-\alpha}{\beta}.
\end{equation}
This distinction between Fourier decay (as with curved manifolds)  and no Fourier decay (as with flat manifolds) is also fundamental in the fractal geometry literature.  Random sets often exhibit Fourier decay (see e.g.~\cite{Sal51}), whereas  sets with arithmetic structure often do not (see e.g.~\cite{LP22}). Thus, the question asked by Stein naturally extends  to fractals and suggests a large programme of research investigating restriction problems for fractal measures in various contexts.

\subsection{New Stein--Tomas type results using the Fourier spectrum}

In this article we obtain a new range of  $q$ for which  \eqref{eq:restrictionp2} holds in terms of the  Fourier spectrum; a family of dimensions that continuously interpolate between the Fourier and Sobolev dimensions for measures. These dimensions extract more nuanced information about the measures than the Fourier and Sobolev dimensions alone, and this extra information can be put to use to study restriction problems. In Section~\ref{sec:main} we state and prove the main theorem (Theorem~\ref{thm:mainthm}), which follows by an interpolation argument between the $L^2\to L^2$ estimate given in \cite{Moc00} using the Frostman dimension and a new $L^{q'}\to L^q$ estimate obtained directly from the Fourier spectrum, which we state precisely in Corollary~\ref{cor:restrictionSpectrum}.

Theorem~\ref{thm:mainthm}  provides a continuum of Stein--Tomas type estimates, which are optimised to provide our main result.  However, it is worth noting that at one end of this continuum we recover the usual Stein--Tomas theorem and at the other end we obtain a new estimate which can be stated  in terms of the Fourier and \emph{Sobolev} dimensions which typically outperforms Stein--Tomas for multifractal measures, see Corollary~\ref{thm:restrictionSobolev}. We also provide results in the other direction by giving a range of $q$ in terms of the Fourier spectrum for which the $L^2\to L^q$ extension  estimate \eqref{eq:restrictionp2} fails, generalising an observation  of Hambrook and {\L}aba, see Theorem~\ref{converse}. In Section~\ref{sec:endpoint} we give a restriction estimate for Lorentz spaces using the Fourier spectrum, which allows us to prove the endpoint estimate, Theorem~\ref{thm:endpoint}, of Theorem~\ref{thm:mainthm}. We illustrate our results with several examples, including the surface measure on the cone and the moment curve (where we can beat the Stein--Tomas estimate) and several fractal measures; see  Sections~\ref{sec:cone}--\ref{sec:fractals}. 

In order to apply our results to the surface measure on the cone, we first explicitly derive the Fourier spectrum for this measure (see Proposition \ref{thm:coneSpectrum}); a result which  may be of independent  interest. We also give the Fourier spectrum of arclength measure on the moment curve (see Proposition \ref{prop:moment}), and this example exhibits multiple phase transitions which has  not been previously observed in `natural examples'.

\section{Preliminaries}

\subsection{Frostman, Hausdorff and Fourier dimensions}

We begin with a short summary of the different definitions of fractal dimensions that we will use. For more details we refer the reader to the book \cite{Mat15} and the article \cite{Fra24}. Unless stated otherwise, throughout the article we shall work with non-zero, finite, compactly supported, Borel measures on $\rd$.

Given $0<\alpha\leq d$, a measure $\mu$ satisfies the Frostman condition with exponent $\alpha$ if for all $r>0$ and $x\in\rd$, $\mu( B(x,r) )\lesssim r^\alpha$. If a measure $\mu$ satisfies this condition, its $s$-energy is finite for all $s<\alpha$, i.e.
\begin{equation*}
    I_{s}(\mu)\coloneqq \iint|x-y|^{-s}\,d\mu(x)\,d\mu(y)< \infty.
\end{equation*}
For an integrable function $f:\rd\to\C$, its Fourier transform is defined as
\begin{equation*}
    \widehat{f}(\xi) = \int_{\rd} e^{-2\pi i \xi\cdot x}f(x)\,dx,
\end{equation*}
and this operator can be easily extended to functions in $L^p$ for $p\in[1,2]$. The Fourier transform of a measure $\mu$ is
\begin{equation*}
    \widehat{\mu}(\xi) = \int_{\rd}e^{-2\pi i \xi\cdot x}\,d\mu(x).
\end{equation*}
Note that for $0<s< d$, it is a consequence of Parseval's theorem that
\begin{equation}\label{eq:senergy}
    I_{s}(\mu)\approx_{d,s} \int_{\rd}\big| \widehat{\mu}(\xi) \big|^{2} |\xi|^{s-d}\,d\xi.
\end{equation}

We define the Frostman and Sobolev dimensions of a measure as
\begin{equation*}
    \frd\mu = \sup\big\{ \alpha\in\R : \mu( B(x,r) )\lesssim r^\alpha,~~\forall r>0,x\in\rd \big\};
\end{equation*}
and
\begin{equation*}
    \sd\mu = \sup\{ s\in\R : \int_{\rd}\big| \widehat{\mu}(\xi) \big|^{2} |\xi|^{s-d}\,d\xi<\infty \},
\end{equation*}
respectively. Frostman's lemma implies that $\frd\mu\leq \sd\mu$ and when $\mu$ is supported on a null set $ \sd\mu \leq d$, in which case the Sobolev dimension is also called the $L^2$-dimension, the energy dimension, or the correlation dimension. Moreover,  the Hausdorff dimension of a Borel set $X\subseteq\rd$ is the supremum of $\min\{ d,\sd\mu \}$, taken over all measures on $X$.

The finiteness of $I_{s}(\mu)$ for $s \in (0,d)$ says that, on average, $\big| \widehat{\mu}(\xi) \big|$ decays like $|\xi|^{-\frac{s}{2}}$. The supremum of the polynomial rate of decay of $\big| \widehat{\mu}(\xi)\big|$ is its Fourier dimension,
\begin{equation*}
    \fd\mu = \sup\big\{ s\in\R : \sup_{\xi\in\rd}\big| \widehat{\mu}(\xi) \big|^2 |\xi|^{s}<\infty \big\}.
\end{equation*}
It is easy to see that $\fd\mu\leq\sd\mu$. What is more, in \cite[Section~3]{Mit02} the author shows that
\[
\min\Big\{ \frac{\fd\mu}{2}, d \Big\}\leq\frd\mu,
\]
 and this bound is sharp; see \cite{LL24+}. We say that a measure $\mu$ is Salem if $\fd\mu  = \sd\mu$. 

By taking appropriate limits, the Stein--Tomas restriction theorem \eqref{eq:SteinTomas} implies that  \eqref{eq:restrictionp2} holds for all
\[
q > 2+4 \frac{d-\frd\mu}{\fd\mu}.
\]
This version---stated in terms of dimensions---is useful to keep in mind when comparing with our main theorem.

\subsection{The Fourier spectrum}

The Fourier and Sobolev dimensions quantify the rate of decay of $\big| \widehat{\mu}(\xi) \big|$ in a weighted $L^\infty$ and $L^2$ sense, respectively. If instead we consider a weighted $L^p$ decay, we arrive to the definition of the Fourier spectrum; a family of dimensions first defined in \cite{Fra24} that interpolate between the Fourier and Sobolev dimensions.

For $\theta\in(0,1]$, define the $(s,\theta)$-energy of a measure $\mu$ as
\begin{equation*}
    \J_{s,\theta}(\mu)\coloneqq \bigg(\int_{\rd}\big| \widehat{\mu}(\xi) \big|^{\frac{2}{\theta}}|\xi|^{\frac{s}{\theta}-d}\,d\xi\bigg)^\theta
\end{equation*}
and for $\theta = 0$,
\begin{equation*}
    \J_{s,0}(\mu) = \sup_{\xi\in\rd}\big| \widehat{\mu}(\xi) \big|^{2}|\xi|^{s}.
\end{equation*}
Then the Fourier spectrum of $\mu$ at $\theta$ is
\begin{equation*}
    \fs\mu = \sup\{ s\in\rd : \J_{s,\theta}(\mu)<\infty \}.
\end{equation*}

It is immediate from its definition that $\fd^0\mu = \fd\mu$ and $\fd^1\mu = \sd\mu$. Furthermore, \cite[Theorem~1.1]{Fra24} states that $\fs\mu$ is non-decreasing, concave, and, for compactly supported measures, continuous for all $\theta\in[0,1]$.  In fact, it was proved in \cite[Proposition 4.2]{CFdO24} that for compactly supported $\mu$, $\fs \mu \leq \fd \mu + d \theta$ for  all $\theta \in [0,1]$.

The Fourier spectrum has already provided several applications where one uses the additional information provided by the spectrum of dimensions.  These applications include contributions to the Falconer distance problem \cite[Section 7]{Fra24}, sumset type problems \cite[Section 6]{Fra24}, and the dimension theory of orthogonal projections \cite{FdO24}.

\section{Main results} \label{sec:main}

\subsection{$L^2$ restriction estimates from the Fourier spectrum}

In the following theorem we give a new range of $q$ for the $L^{q'} \to L^2$ Fourier restriction   estimate to hold. The proof follows that of \cite{Moc00} where instead of using the Fourier dimension of $\mu$, we use its Fourier spectrum, which allows us to interpolate between the known $L^2\to L^2$ estimate given by the Frostman condition with a new $L^{q'}\to L^q$ estimate given by the Fourier spectrum, where we use the information given by the $(s,\theta)$-energies of a discretisation of $\mu$ at different scales.

\begin{thm}\label{thm:mainthm}
  Let $\mu$ be a non-zero, finite, compactly supported, Borel measure on $\rd$, with $\frd\mu=\alpha$ for some $0<\alpha< d$. If
  \begin{equation*}
    q>2 + 2\inf\limits_{\substack{\theta\in[0,1] \\ \fs\mu>d\theta}}\frac{(d-\alpha)(2-\theta)}{\fs\mu-\alpha \theta},
  \end{equation*}
  then for all $f\in L^2(\mu)$,
  \begin{equation*}
      \|\widehat{f\mu}\|_{L^q(\rd)} \lesssim \|f\|_{L^2(\mu)}. 
  \end{equation*}

  Equivalently, if 
  \begin{equation*}
    1 \leq q' < 1 + \sup_{\substack{\theta\in[0,1] \\ \fs\mu>d\theta}} \frac{\fs\mu-\alpha\theta}{\fs\mu + 4(d-\alpha) - \theta(2d-\alpha)},
  \end{equation*}
  then for all $f\in L^{q'}(\rd)$,
  \begin{equation*}
      \|\widehat{f}\,\|_{L^2(\mu)}\lesssim \|f\|_{L^{q'}(\rd)}.
  \end{equation*}
\end{thm}

In fact, we can obtain the appropriate endpoint estimate. Although the following theorem implies Theorem~\ref{thm:mainthm}, its proof is  much more involved, since it depends on restriction estimates for Lorentz spaces. For this reason we postpone its proof to Section~\ref{sec:endpoint}, and give a direct proof of Theorem~\ref{thm:mainthm} at the end of this section.
\begin{thm}\label{thm:endpoint}
  Let $0<\alpha< d$, $\theta\in[0,1]$, $d\theta<s$, and let $\mu$ be a non-zero, finite, compactly supported, Borel measure on $\rd$, such that
  \begin{equation*}
       \mu\big( B(x,r) \big)\lesssim r^{\alpha},\quad \quad \forall x\in\rd,\,r>0;
  \end{equation*}
  and
  \begin{equation*}
      \J_{s,\theta}(\mu)< \infty.
  \end{equation*}
  Let
  \begin{equation*}
    q_{0} = 2 + 2\frac{(d-\alpha)(2-\theta)}{s-\alpha \theta}.
  \end{equation*}
  Then for all $f\in L^2(\mu)$,
  \begin{equation*}
      \|\widehat{f\mu}\|_{L^{q_{0}}(\rd)} \lesssim \|f\|_{L^2(\mu)}. 
  \end{equation*}

  Equivalently, let
  \begin{equation*}
    q_{0}' = 1 + \frac{s-\alpha\theta}{s + 4(d-\alpha) - \theta(2d-\alpha)}. 
  \end{equation*}
  Then for all $f\in L^{q_{0}'}(\rd)$,
  \begin{equation*}
    \|\widehat{f}\,\|_{L^2(\mu)} \lesssim  \|f\|_{L^{q_{0}'}(\rd)}.
  \end{equation*}
\end{thm}

The conditions $\fs\mu>d\theta$ and $s>d\theta$ are necessary when taking the infimum in Theorem~\ref{thm:mainthm} and in Theorem~\ref{thm:endpoint}. Indeed, the middle third Cantor measure has $\fs\mu \leq \theta$ for all $\theta$ (see \cite[Corollary~4.4]{Fra24}),  but if we attempt to apply Theorem~\ref{thm:mainthm} or Theorem~\ref{thm:endpoint} without the condition we would get a non-trivial restriction range because $\dim_\textup{F}^{1/2} \mu > \alpha/2$  (see \cite[Corollary~4.4]{Fra24}).  However, this is not possible by the following lemma, described to us by Chun-Kit Lai. In particular, the Fourier transform of the middle third Cantor measure does not decay to 0 at $\infty$.
\begin{lma}\label{lma:decay}
  Let $\mu$ be a non-zero, finite, compactly supported, Borel measure on $\rd$. If for some $q>1$, $\|\widehat{f\mu}\|_{L^q(\rd)}\leq \|f\|_{L^2(\mu)}$ for all $f\in L^2(\mu)$, then $\big| \widehat{\mu}(\xi) \big|\to 0$ as $|\xi|\to \infty$.
\end{lma}
\begin{proof}
  Taking $f=1$, and using that  $\mu$ has compact support,  $f\in L^2(\mu)$ and so $\widehat{\mu}\in L^q(\rd)$.  However, since  $\widehat{\mu}$ is uniformly continuous (even uniformly Lipschitz, see \cite[(3.19)]{Mat15}) this implies that $\widehat{\mu}$ must decay at infinity.
\end{proof}

In the proof of Theorem~\ref{thm:mainthm} (see Section~\ref{sec:proofMainThm}) we obtained the following restriction estimate directly from the Fourier spectrum, without appealing to the $L^2\to L^2$ estimate coming from the Frostman exponent.  We record it here for interest. 
\begin{cor}\label{cor:restrictionSpectrum}
  Let $\mu$ be a non-zero, finite, compactly supported, Borel measure on $\rd$. If
  \begin{equation*}
    q>\inf \Big\{ \frac{4}{\theta} : \theta\in[0,1],\, \fs\mu >d\theta \Big\},
  \end{equation*}
  then for all $f\in L^2(\mu)$,
  \begin{equation*}
      \|\widehat{f\mu}\|_{L^q(\rd)} \lesssim \|f\|_{L^2(\mu)}. 
  \end{equation*}
  
  Equivalently, if 
  \begin{equation*}
    1 \leq q' < \sup\Big\{ \frac{4}{4-\theta} : \theta\in[0,1],\, \fs\mu >d\theta \Big\},
  \end{equation*}
  then for all $f\in L^{q'}(\rd)$,
  \begin{equation*}
      \|\widehat{f}\|_{L^2(\mu)}\lesssim \|f\|_{L^{q'}(\rd)}.
  \end{equation*}
\end{cor}

As a simple consequence of Theorem~\ref{thm:endpoint} and real interpolation with the trivial $L^1\to L^\infty$ estimates
\begin{equation*}
    \|\widehat{f\mu}\|_{L^{\infty}(\rd)}\lesssim \|f\|_{L^{1}(\mu)}\quad \text{and}\quad \|\widehat{f}\,\|_{L^{\infty}(\mu)}\lesssim \|f\|_{L^1(\rd)},
\end{equation*}
we obtain the following corollary.
\begin{cor}
  Let $0<\alpha< d$, $\theta\in[0,1]$, $d\theta<s$, and let $\mu$ be a non-zero, finite, compactly supported, Borel measure on $\rd$, such that
  \begin{equation*}
       \mu\big( B(x,r) \big)\lesssim r^{\alpha},\quad \quad \forall x\in\rd,\,r>0;
  \end{equation*}
  and
  \begin{equation*}
      \J_{s,\theta}(\mu)< \infty.
  \end{equation*}
  For $q\geq q_{0}$, where $q_{0}$ is as defined in Theorem~\ref{thm:endpoint}, the estimate
  \begin{equation*}
      \|\widehat{f\mu}\|_{L^q(\rd)}\lesssim \|f\|_{L^{2q/(2q-q_{0})}(\mu)}
  \end{equation*}
  holds for all $f\in L^{\frac{2q}{2q-q_{0}}}(\mu)$.

  Equivalently, for $1\leq q' \leq q_{0}'$,
  \begin{equation*}
    \|\widehat{f}\,\|_{L^{2q/q_{0}}(\mu)} \lesssim  \|f\|_{L^{q'}(\rd)}
   \end{equation*}
   holds for all $f\in L^{q'}(\rd)$.
\end{cor}

$L^2$ restriction theorems of the type proved in this section have numerous applications.  For example, in \cite[Theorem 5.1]{Moc00} Mockenhaupt uses the Stein--Tomas theorem to prove that certain convolution operators are $L^p$-multipliers, with the range of $p$ depending on the established range for $L^2$ restriction.  One can do the same using our Theorem \ref{thm:mainthm} and obtain improved ranges, but we leave the details to the reader.

\subsubsection{Proof of Theorem~\ref{thm:mainthm}} \label{sec:proofMainThm}
  We will show that $ \|\widehat{\mu}*f\|_{L^{q}(\rd)}\lesssim \| f\|_{L^{q'}(\rd)}$ for every $f\in L^{q'}(\rd)$ for all $q$ sufficiently close to the desired threshold.  This is equivalent to both statements in the theorem (see Section~\ref{sec:restrictionIntro}). In particular, establishing \eqref{eq:restrictionp2} for some $q$ establishes it for all larger $q$ since $\mu$ is compactly supported.

To begin with, let $q \geq 1$ be arbitrary.  We start by breaking up $\widehat{\mu}$ dyadically in the standard way. Let $\rho\in\mathcal{C}^\infty(\rd)$ be such that $\rho(\xi)=0$ for $|\xi|\leq 1/2$ and $\rho(\xi) = 1$ for $|\xi|\geq1$. Define the function $\varphi(\xi) = \rho(2\xi) - \rho(\xi)$, noting that  $\spt\varphi\subset\{ \xi\in\rd : \frac{1}{4}\leq|\xi|\leq1 \}$ and
\[
\sum_{j=0}^\infty \varphi(2^{-j}\xi)  = 1 \qquad \text{(for $|\xi| \geq 1$).}
\]
  Clearly $\widehat{\mu}= \Phi+\sum_{j=0}^\infty \widehat{\mu}_{j}$ where
\[
\Phi(\xi) =  \Big( 1 - \sum_{j=0}^\infty \varphi(2^{-j}\xi) \Big)\widehat{\mu}(\xi) 
\]
and  $\widehat{\mu}_{j}(\xi) =\varphi(2^{-j}\xi)\widehat{\mu}(\xi)$.  Therefore,
  \begin{equation}\label{eq:boundmuq}
      \|\widehat{\mu}*f\|_{L^q(\rd)}\leq \| \Phi* f  \|_{L^q(\rd)} +  \sum_{j=0}^\infty \| \widehat{\mu}_{j}* f \|_{L^q(\rd)}.
  \end{equation}
  For the first term, since $\Phi$ is bounded and has compact support,  
\[
\| \Phi\|_{L^{q/2}(\rd)} \lesssim \| \Phi \|_{L^\infty(\rd)}\lesssim 1.
\]
Using this and Young's inequality for convolutions gives,
  \begin{equation*}
   \| \Phi* f  \|_{L^q(\rd)} \leq   \| \Phi   \|_{L^{q/2}(\rd)} \| f \|_{L^{q'}(\rd)} \lesssim \|f\|_{L^{q'}(\rd)}.
  \end{equation*}
  For the second term we will obtain the $L^{q'}\to L^q$ bound by interpolation. Let $\theta\in (0,1]$ and $s >0$  be such that $\fs\mu >s> d\theta$.  Since $\J_{s,\theta}(\mu)\lesssim1$,
  \begin{equation*}
    \|\widehat{\mu}_{j}\|_{L^{2/\theta}(\rd)} \approx \bigg( \int_{|\xi|\approx 2^{j}}\big| \widehat{\mu}(\xi) \big|^{\frac{2}{\theta}} |\xi|^{\frac{s}{\theta}-d}\,2^{-j \big( \frac{s}{\theta}-d \big)}\,d\xi\bigg)^\frac{\theta}{2} \lesssim 2^{-j \frac{s-d\theta}{2} }\J_{s,\theta}(\mu)^{1/2} \lesssim 2^{-j \frac{s-d\theta}{2} },
  \end{equation*}
  and so by Young's inequality for convolutions,
  \begin{equation}\label{eq:boundmuj}
      \|\widehat{\mu}_{j}*f\|_{L^{4/\theta}(\rd)}\leq \|\widehat{\mu}_{j}\|_{L^{2/\theta}(\rd)}\|f\|_{L^{4/(4-\theta)}(\rd)}\lesssim 2^{-j \frac{s-d\theta}{2}}\|f\|_{L^{4/(4-\theta)}(\rd)}.
  \end{equation}
  Following \cite[Theorem~4.1]{Moc00} we can use the fact that $\mu$ is $\alpha_0$-Frostman for  $\alpha_0<\alpha$ to prove that
  \begin{equation}\label{eq:boundL2}
      \|\widehat{\mu}_{j}*f\|_{L^2(\rd)}\lesssim 2^{-j(\alpha_0-d)}\|f\|_{L^2(\rd)}.
  \end{equation}
  Let $\lambda \in [0,1]$ and apply the Riesz--Thorin interpolation theorem with parameter $\lambda$ between \eqref{eq:boundmuj} and \eqref{eq:boundL2}.  This gives a $L^{q'}\to L^q$ estimate, with 
  \begin{equation*}
\frac{1}{q} = \frac{(1-\lambda)\theta}{4} + \frac{\lambda}{2},
  \end{equation*}
noting that this forces $2 \leq q \leq 4/\theta$. More precisely, we obtain
  \begin{equation*}
      \|\widehat{\mu}_{j}*f\|_{L^{q}(\rd)}\leq 2^{-j\big( (1-\lambda)\frac{s-d\theta}{2}+\lambda(\alpha_0 - d) \big)}\|f\|_{L^{q'}(\rd)}.
  \end{equation*}
Therefore, provided 
  \begin{equation} \label{eq:cond}
(1-\lambda)\frac{s-d\theta}{2}+\lambda(\alpha_0 - d)>0,
  \end{equation}
we get
  \begin{equation*}
      \sum_{j=0}^\infty \|\widehat{\mu}_{j}*f\|_{L^q(\rd)}\lesssim \|f\|_{L^{q'}(\rd)}.
  \end{equation*}
  Using this in \eqref{eq:boundmuq} establishes the desired  $L^{q'}\to L^q$  bound whenever $q \in [2,4/\theta]$ satisfies  \eqref{eq:cond}, that is, whenever
  \begin{equation*}
    \frac{4}{\theta} \geq q>2 + 2 \frac{(d-\alpha_0)(2-\theta)}{s-\alpha_0 \theta}.
  \end{equation*}
Crucially, using that $s>d \theta$ and $\alpha_0<d$, 
  \begin{equation*}
    \frac{4}{\theta}  >2 + 2 \frac{(d-\alpha_0)(2-\theta)}{s-\alpha_0\theta}
  \end{equation*}
and so there is a non-empty range of possible $q$.    This proves the desired result upon letting $\alpha_0 \to \alpha$ and $s \to \fs \mu$.

\subsection{A partial converse from the Fourier spectrum}

Another interesting question is to determine when the restriction estimate \eqref{eq:restriction} does not hold.  In \cite{HL13} the authors noted that \eqref{eq:restriction} will not hold for any $2 \leq q<\frac{2d}{\sd\mu}$, which is a non-trivial statement when $\sd \mu <d$.  We repeat their simple argument here. Let  $d>s>\sd\mu$, in which case $I_{s}(\mu)=\infty$. Then, by H\"older's inequality with conjugate exponents $\frac{q}{2}$ and $\frac{q}{q-2}$,
\begin{equation*}
    \infty =\int_{|\xi| \geq 1} |\widehat \mu(\xi)|^2 |\xi |^{s-d} d \xi  \leq \|\widehat{\mu}\|_{q}^2 \bigg( \int_{|\xi| \geq 1} |\xi|^{\frac{q(s-d)}{q-2}} \,d\xi \bigg)^{\frac{q-2}{q}},
\end{equation*}
and the last integral is finite if $q<\frac{2d}{s}$.  Therefore, for $q<\frac{2d}{\sd\mu}$, $\widehat \mu$ cannot be in $L^q$ and so \eqref{eq:restriction} does not hold for any $p$.  We can generalise this result (which we refer to later as the Hambrook--{\L}aba lower bound) using the Fourier spectrum.

\begin{thm} \label{converse}
  Let $\mu$ be a non-zero, finite, compactly supported, Borel measure on $\rd$. If $q=\frac{2}{\theta}$ for some $\theta \in (0,1]$ such that $ \mathcal{J}_{d\theta, \theta}(\mu)=\infty$ and $p \in [1,\infty]$, then  for $f =1$, 
  \begin{equation*}
      \|\widehat{f\mu}\|_{L^q(\rd)} =\infty > \|f\|_{L^p(\mu)} = \mu(\mathbb{R}^d)^{1/p}.
  \end{equation*}
Hence, \eqref{eq:restriction} does not hold for any $p \in [1,\infty]$.

In particular, \eqref{eq:restriction} does not hold for any $p\in[1,\infty]$ if
  \begin{equation*}
      q< \sup\Big\{ \frac{2}{\theta} : \fs\mu < d\theta \Big\}.
  \end{equation*}
\end{thm}

\begin{proof}
Let  $f=1$ and  $\theta \in (0,1]$ be such that $ \mathcal{J}_{d\theta, \theta}(\mu)=\infty$. Then
\[
\infty = \mathcal{J}_{d\theta, \theta}(\mu) = \bigg(\int |\widehat{\mu}(\xi) |^{\frac{2}{\theta}}  \, d\xi \bigg)^\theta =      \|\widehat{f\mu}\|_{L^{2 /\theta}{(\rd)}}^2,
\]
as required.
\end{proof}

Note that for all $\theta\in(0,1]$ such that $\fs\mu < d\theta$, $\mathcal{J}_{d\theta, \theta}(\mu)=\infty$. If $\sd \mu = \dim_\textup{F}^1 \mu <d$, then we recover the fact that \eqref{eq:restriction} fails for $q<\frac{2d}{\sd\mu}$. However, we will often get a stronger result (see Figure~\ref{fig:restrictionCone}).  Let us consider the smallest possible range provided by the previous theorem.  This will happen when the Fourier spectrum stays above $d\theta$ for as large $\theta$ as possible. Since the Fourier spectrum is increasing in $\theta$, this will happen when $\fs \mu =\sd \mu$ for $\theta \geq \sd \mu /d$. For non-Salem measures $\mu$ this is quite unlikely to happen, and so we will obtain a strictly better (negative) range.  On the other hand, the best information we can get from the above result is when the Fourier spectrum is as small as possible.  Since the Fourier spectrum is concave this will happen when $\fs \mu$ is affine and given by
\[
\fs \mu = \fd \mu + \theta(\sd \mu - \fd \mu).
\]
In this case we get that \eqref{eq:restriction} fails for 
\[
q < 2+ \frac{2(d-\sd \mu)}{\fd \mu}.
\]
Theorem \ref{converse} seems to provide useful new information concerning the general restriction problem.  For example, for the sphere, it gives that  \eqref{eq:restriction}  fails for
\[
q<\frac{2d}{d-1}
\]
and this is one of the sharp thresholds from the restriction conjecture. This threshold is also obtained by the Hambrook--{\L}aba bound.  However, we will see this behaviour again later when we consider the cone and the moment curve and,  in fact, Theorem \ref{converse} gives a sharp threshold for the general restriction problem for these examples whereas  the Hambrook--{\L}aba bound does not (see Sections~\ref{sec:cone} and \ref{sec:moment}). 

Another simple consequence of the above is the following mild strengthening of Lemma \ref{lma:decay}.  

\begin{cor}\label{lma:decay2}
  Let $\mu$ be a non-zero, finite, compactly supported, Borel measure on $\rd$. If  $ \mathcal{J}_{d\theta, \theta}(\mu)=\infty$ for all $\theta \in (0,1]$, then  \eqref{eq:restriction} does not hold for any pairs $p,q \in [1,\infty]$ with $q<\infty$. In particular, this is true  if $\fs \mu < d \theta$ for all $\theta \in (0,1]$ or if  $\widehat{\mu}$ does not decay at infinity.
\end{cor}

\section{Restriction estimates for Lorentz spaces and the endpoint}\label{sec:endpoint}

\subsection{Lorentz spaces and interpolation}

Lorentz spaces arise in the interpolation between $L^1$ and $L^\infty$, and provide a more nuanced description of the behaviour of functions and their distributions. For this reason, they are useful to obtain sharper inequalities involving $L^p$ norms, leading to improvements on classical results such as the Hausdorff--Young inequality, and the Sobolev embedding theorem, among others (see \cite[Chapter~VI \S7.10]{Ste93}). In particular, they have been useful to obtain improved $L^{q'}\to L^{p'}$ restriction estimates (see e.g. \cite{BOS09, BS11,GXZ24, KT98}).

In this section we introduce restriction estimates between Lorentz spaces as a means to prove Theorem~\ref{thm:endpoint}. For more details on Lorentz spaces we refer the reader to \cite{Gra14,SW71}.

Recall that for operators $T$ acting on Lorentz spaces, by density of simple functions and the fact that $L^{q,\infty}$ is a normed space, being of restricted weak type $(p,q)$ is equivalent to the fact that $T$ maps $L^{p,1}(\rd)$ to $L^{q,\infty}(\rd)$ boundedly.


 Given two vector spaces compatible in the sense of interpolation theory (see \cite[Chapter~V]{SW71}), and $\lambda\in[0,1]$, we write $B_{\lambda,\infty}(B_{0},B_{1})$ to denote the real interpolation space between $B_{0}$ and $B_{1}$. We shall only make use of this when $B_{0} = L^p(\rd)$ and $B_{1} = L^q(\rd)$, for some $1\leq p,q\leq \infty$, in which case $B_{\lambda,\infty}(B_{0},B_{1}) = L^{r,\infty}(\rd)$ for $\frac{1}{r} = \frac{1-\lambda}{p} + \frac{\lambda}{q}$. The following lemma is an interpolation theorem stated implicitly by Bourgain in \cite{Bou85}, which can also be found in \cite[Section~6.2]{CSWW99} in a more general setting. We state it here for convenience, in the form that we will use.

\begin{lma}[Bourgain's interpolation trick]\label{lma:Bourgainstrick}
  Let $A_{0}, A_{1}$, and $B_{0},B_{1}$ be two pairs of quasinormed spaces compatible in the sense of interpolation theory.
  If $T_{j}$ is a sequence of operators that satisfy that there exist constants $\alpha_{0}<0<\alpha_{1}$, and $M_{0},M_{1}>0$ such that for all $j\in\N$, $i=0,1$, and $f\in A_{0}\cap A_{1}$,
  \begin{equation*}
      \|T_{j}f\|_{B_{i}}\leq M_{i}2^{j\alpha_{i}}\|f\|_{A_{i}}.
  \end{equation*}
  Then for $\lambda = \frac{\alpha_{0}}{\alpha_{0} - \alpha_{1}}\in(0,1)$,
   and all $f\in A_{0}\cap A_{1}$, the operator $T = \sum_{j=1}^\infty T_{j}$ satisfies
  \begin{equation*}
      \|Tf\|_{B_{\lambda,\infty}(B_{0},B_{1})}\lesssim_{\alpha_{0},\alpha_{1}} M_{0}^{1-\lambda}M_{1}^{\lambda}\|f\|_{A_{0}}^{1-\lambda}\|f\|_{A_{1}}^\lambda.
  \end{equation*}
\end{lma}

\subsection{Restriction estimates for Lorentz spaces}

When considering restriction estimates in Lorentz spaces, the question becomes for which $p,q,r,s\in[1,\infty]$ does it hold that
\begin{equation}\label{eq:restrictionLorentz}
    \|\widehat{f}\,\|_{L^{r',s'}(\mu)}\lesssim \|f\|_{L^{p',q'}(\rd)},
\end{equation}
for all $f\in L^{p',q'}(\rd)$. Similarly to the $L^p$ case (see Section~\ref{sec:restrictionIntro}), by the dual characterisation of Lorentz quasinorms, the $L^{p',q'}\to L^{r',s'}$ restriction estimate \eqref{eq:restrictionLorentz} is equivalent to the $L^{r,s}\to L^{p,q}$ extension estimate
\begin{equation}\label{eq:extensionLorentz}
    \|\widehat{f\mu}\|_{L^{p,q}(\rd)}\lesssim \|f\|_{L^{r,s}(\mu)}.
\end{equation}
What is more, in the case $r=s=2$, both \eqref{eq:restrictionLorentz} and \eqref{eq:extensionLorentz} are equivalent to
\begin{equation*}
    \|\widehat{\mu}*f\|_{L^{p,q}(\rd)}\lesssim \|f\|_{L^{p',q'}(\rd)}.
\end{equation*}

Theorem~\ref{thm:endpoint} is a consequence of the following restriction estimate for Lorentz spaces, which is also of independent interest.

\begin{prop}\label{prop:restrictionLorentz}
  Let $0<\alpha< d$, $\theta\in[0,1]$, $d\theta<s$, and let $\mu$ be a non-zero, finite, compactly supported, Borel measure on $\rd$, such that
  \begin{equation*}
       \mu\big( B(x,r) \big)\lesssim r^{\alpha},\quad \quad \forall x\in\rd,\,r>0;
  \end{equation*}
  and
  \begin{equation*}
      \J_{s,\theta}(\mu)< \infty.
  \end{equation*}
  Let
  \begin{equation*}
      \rho = \frac{4( 2\alpha - s + d(\theta-2) )(\alpha-s+d(\theta-1))}{2s^2 - 2\alpha^2(\theta-4)-ds(7\theta -12)-d^2(14\theta -5\theta^2-8) - \alpha(\theta-4)(-3s + d(3\theta -4))};
  \end{equation*}
  \begin{equation*}
      \sigma = \frac{4(\alpha - s + d(\theta -1))}{\theta(d+\alpha)-2s}.
  \end{equation*}
  Then for any $\rho<p<\sigma'$, and $r>0$ with $\frac{1}{p}-\frac{1}{r} = \frac{(d-\alpha)(2-\theta)}{2(d-\alpha)+s-d\theta}$, and $q\in[1,\infty]$,
  \begin{equation}\label{eq:propRestrictionLorentz}
    \|\widehat{\mu}*f\|_{L^{r,q}(\rd)} \lesssim \|f\|_{L^{p,q}(\rd)}.
  \end{equation}
  What is more, \eqref{eq:propRestrictionLorentz} holds for $r = q_{0}$ and $p = q_{0}'$, where $q_{0}$ is as in Theorem~\ref{thm:endpoint}.
\end{prop}

\subsection{Proof of Theorem~\ref{thm:endpoint}}
  First, by Parseval's formula,
  \begin{equation}\label{eq:Plancherel}
      \|\widehat{f}\,\|_{L^2(\mu)}^2 = \|\widehat{\overline{f}}\,\|_{L^2(\mu)}^2 = \int_{\rd} \overline{f}(\xi)(\widehat{\mu}*f)(\xi)\,d\xi.
  \end{equation}
  H\"older's inequality for Lorentz spaces with $1 = \frac{1}{q_{0}'} + \frac{1}{q_{0}}$ and $1 = \frac{1}{2} + \frac{1}{2}$, and the final part of Proposition~\ref{prop:restrictionLorentz} with $q=2$, $r = q_{0}$, and $p = q_{0}'$ gives
  \begin{equation*}
    \|\widehat{f}\,\|_{L^2(\mu)}^2 = \|\widehat{f}\,\|_{L^{2,2}(\mu)}^2 \leq \|f\|_{L^{q_{0}',2}(\rd)}\,\|f*\widehat{\mu}\|_{L^{q_{0},2}(\rd)} \lesssim \|f\|_{L^{q_{0}',2}(\rd)}^2.
  \end{equation*}
  Since $q_{0}'< 2$, $\|f\|_{L^{q_{0}',2}(\rd)} \leq \|f\|_{L^{q_{0}'}(\rd)}$, and the result follows.

\subsection{Proof of Proposition~\ref{prop:restrictionLorentz}}

Let $\theta\in(0,1]$ and $s>d\theta$ be such that $\J_{s,\theta}(\mu)<\infty$. By following the proof of Theorem~\ref{thm:mainthm} we get the following estimates for the dyadic decomposition $(\mu_{j})_{j}$ of $\mu$ (see \eqref{eq:boundmuj} and \eqref{eq:boundL2}):
  \begin{equation}\label{eq:boundmujEndpoint}
      \|\widehat{\mu}_{j}*f\|_{L^{4/\theta}(\rd)}\lesssim 2^{-j \frac{s-d\theta}{2}}\|f\|_{L^{4/(4-\theta)}(\rd)};
  \end{equation}
  and 
  \begin{equation}\label{eq:boundL2Endpoint}
      \|\widehat{\mu}_{j}*f\|_{L^2(\rd)}\lesssim 2^{j(d-\alpha)}\|f\|_{L^2(\rd)}.
  \end{equation}

  Since $s>d\theta$, we can use Bourgain's interpolation trick between the estimates \eqref{eq:boundmujEndpoint} and \eqref{eq:boundL2Endpoint}. For this define
\begin{equation*}
  \lambda = \frac{s-d\theta}{2(d-\alpha)+s-d\theta}.
\end{equation*}
  Let
  \begin{equation*}
      q_{0} = 2 + 2\frac{(d-\alpha)(2-\theta)}{s-\alpha \theta}
  \end{equation*}
  noting that
  $\frac{(1-\lambda)\theta}{4} + \frac{\lambda}{2} = \frac{1}{q_{0}}$, implies that $B_{\lambda,\infty}\big(L^{4/\theta}(\rd),L^{2}(\rd)\big) = L^{q_{0},\infty}(\rd)$. Hence, applying Lemma~\ref{lma:Bourgainstrick} with $f = 1_{E}$ for an arbitrary measurable set $E$,
  \begin{equation*}
    \|f\|_{L^{4/(4-\theta)}(\rd)}^{1-\lambda}\|f\|_{L^{2}(\rd)}^\lambda = |E|^{\frac{1}{q_{0}'}} = \|f\|_{L^{q_{0}',1}(\rd)},
\end{equation*}
  thus, the convolution operator with $\widehat{\mu}$ is of restricted weak type $(q_{0}',q_{0})$. Thus, for all $f\in L^{q_{0}',1}(\rd)$,
  \begin{equation}\label{eq:firstIterpolIneq}
    \|\widehat{\mu}*f\|_{L^{q_{0},\infty}(\rd)} \lesssim \|f\|_{L^{q_{0}',1}(\rd)}.
  \end{equation}

  The second restricted weak type estimate comes from \cite[Proof of Proposition~2.1]{BS11}, where the authors use a special case of \eqref{eq:firstIterpolIneq} to prove that
  \begin{equation}\label{eq:secondIneq}
    \|\widehat{\mu}_{j}*f\|_{L^2(\rd)} \lesssim 2^{j\frac{d-\alpha}{2}}\|f\|_{L^{q_{0}',1}(\rd)}.
  \end{equation}
  
  We now use Bourgain's interpolation trick between \eqref{eq:boundmujEndpoint} and \eqref{eq:secondIneq}. Let
  \begin{equation*}
    \gamma = \frac{s-d\theta}{d-\alpha + s-d\theta}.
\end{equation*}
  Note that for $\rho$ and $\sigma$ as defined in the proposition, $\frac{1}{\rho} = \frac{(1-\gamma)(4-\theta)}{4} + \frac{\gamma}{q_{0}'}$, and $\frac{1}{\sigma} = \frac{(1-\gamma)\theta}{4} + \frac{\gamma}{2}$. Thus, $B_{\gamma,\infty}\big(L^{4/\theta}(\rd),L^2(\rd)\big) = L^{\sigma,\infty}(\rd)$, and letting $f = 1_{E}$ for an arbitrary measurable set $E$,
  \begin{equation*}
      \|f\|_{L^{4/(4-\theta)}(\rd)}^{1-\gamma}\|f\|_{L^{q_{0}',1}(\rd)}^\gamma = |E|^{\frac{1}{\rho}} = \|f\|_{L^{\rho,1}(\rd)},
  \end{equation*}
  which proves that the convolution operator with $\widehat{\mu}$ is of restricted weak type $(\rho,\sigma)$. That is, for all $f\in L^{\rho,1}(\rd)$,
  \begin{equation}\label{eq:secondInterpolIneq}
      \|\widehat{\mu}*f\|_{L^{\sigma,\infty}(\rd)} \lesssim\|f\|_{L^{\rho,1}(\rd)}.
  \end{equation}

  Finally, using \cite[Theorem~1.4.19]{Gra14} to interpolate between \eqref{eq:secondInterpolIneq} and its dual 
  \begin{equation*}
    \|\widehat{\mu}*f\|_{L^{\rho',\infty}(\rd)} \lesssim\|f\|_{L^{\sigma',1}(\rd)},
\end{equation*}
  we get that for any $\rho<p<\sigma'$, $\frac{1}{p}-\frac{1}{r} = \frac{(d-\alpha)(2-\theta)}{2(d-\alpha)+s-d\theta}$, and $q\in[1,\infty]$,  
  \begin{equation*}
      \|\widehat{\mu}*f\|_{L^{r,q}(\rd)} \lesssim \|f\|_{L^{p,q}(\rd)},
  \end{equation*}
  as we wanted to prove. The final part of the result follows from this last interpolation by taking the interpolation parameter in \cite[Theorem~1.4.19]{Gra14} to be $\frac{1}{2}$, since $\frac{1}{q_{0}}$ is the midpoint between $\frac{1}{\sigma}$ and $\frac{1}{\rho'}$, and $\frac{1}{q_{0}'}$ the midpoint between $\frac{1}{\rho}$ and $\frac{1}{\sigma'}$.

\section{Application 1: Improving the  Stein--Tomas range}

In \cite{HL13,HL16,Che16} the authors proved that there is a large family of measures satisfying \eqref{eq:FrostmanExp} and \eqref{eq:FourierDecay}, for which the range given by Stein--Tomas \eqref{eq:SteinTomas} is optimal. However, from the constructions of Chen~\cite{Che14}, Chen--Seeger~\cite{CS17}, Shmerkin--Suomala~\cite{SS18}, and {\L}aba--Wang~\cite{LW18}, we know that there are other measures (some also satisfying \eqref{eq:FrostmanExp} and \eqref{eq:FourierDecay}) for which the range of $q$ such that \eqref{eq:restrictionp2} holds can be improved. We now give a condition on the Fourier spectrum of measures for which the range given by Stein--Tomas is not optimal; see  Figure~\ref{fig:STimprovement}.

The range of $q$ obtained in Theorem~\ref{thm:mainthm} will improve on the Stein--Tomas range \eqref{eq:SteinTomas} if for some $\theta\in[0,1]$
\begin{equation*}
    2 + \frac{2(d-\alpha)(2-\theta)}{\fs\mu - \alpha\theta} < 2 + \frac{4(d-\alpha)}{\fd\mu}
\end{equation*}
and $\fs\mu>d\theta$, where  $\alpha = \frd\mu$. Therefore, we need
\begin{equation}\label{eq:improvement}
    \fs\mu > \max\Big\{ \theta\big(\alpha - \tfrac{\fd\mu}{2}\big) + \fd\mu, d\theta \Big\}.
\end{equation}
We note that by concavity $\fs\mu \geq(\sd\mu -\fd\mu)\theta+\fd\mu$ always holds.

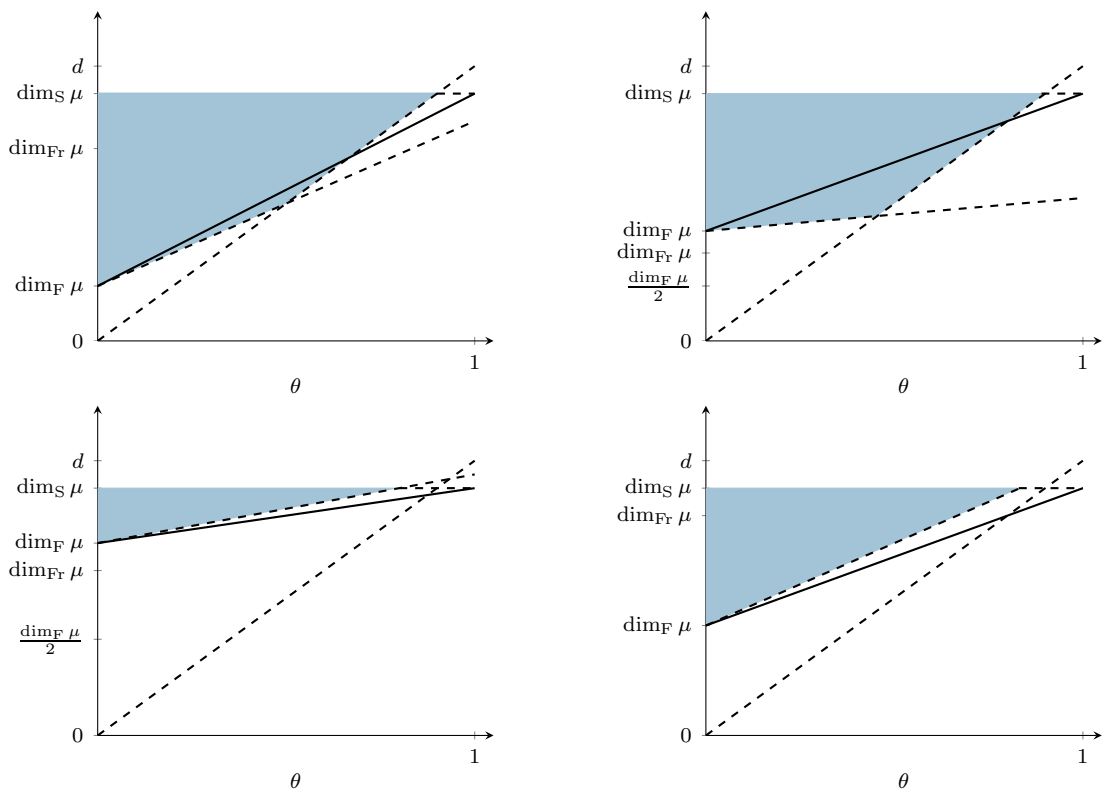
\begin{figure}[H]
  \centering
  \begin{subfigure}{.5\textwidth}
    \centering
    \begin{tikzpicture}[scale = 0.7]
      \def\HD{0.9}
      \def\alpha{0.7}
      \def\FD{0.2}
      \begin{axis}[
          axis lines = left,
          ymin = 0,
          ymax = 1.2,
          xmin = 0,
          xmax = 1.05,
          width=9.1cm,
          xlabel=\scriptsize{$\theta$},
          xtick = {1},
          ytick = {0,\FD, \alpha, \HD,1},
          xticklabels = {\scriptsize{$1$}},
          yticklabels = {\scriptsize{$0$},\scriptsize{$\fd\mu$},\scriptsize{$\frd\mu$},\scriptsize{$\sd\mu$},\scriptsize{$d$}}
      ]
      \addplot [
          domain=0:0.9,
          samples=100, 
          color=plotblue,
          name path=HD,
          thick,
      ]
      {\HD};
      \addplot [
          domain= 0.9:1,
          samples=100, 
          style=dashed,    
          thick,      
      ]
      {\HD};

      \addplot [ 
          domain=0:1, 
          samples=100, 
          style=dashed,
          color=black,
          name path=lower,
          thick,
      ]
      {x*(\alpha - \FD/2) + \FD};
      \addplot [ 
        domain=0:1,
        samples=100,
        style=dashed,
        color = black,
        name path=dtheta,
        thick,
      ]{x};
      
      \addplot [ 
        domain=0:1,
        samples=100,
        color = black,
        thick,
      ]{x*(\HD-\FD) + \FD};

      \addplot[plotblue] fill between[of=lower and HD, soft clip={domain=0:0.51}];
      \addplot[plotblue] fill between[of=dtheta and HD, soft clip={domain=0.5:0.95}];
      \end{axis}
    \end{tikzpicture}
  \end{subfigure}%
  \begin{subfigure}{.5\textwidth}
    \centering
    \begin{tikzpicture}[scale = 0.7]
      \def\HD{0.9}
      \def\alpha{0.32}
      \def\FD{0.4}
      \def\halfFD{0.2}
      \begin{axis}[
          axis lines = left,
          ymin = 0,
          ymax = 1.2,
          xmin = 0,
          width=9.1cm,
          xmax = 1.05,
          xlabel=\scriptsize{$\theta$},
          xtick = {1},
          ytick = {0,\halfFD,\FD, \alpha, \HD,1},
          xticklabels = {\scriptsize{$1$}},
          yticklabels = {\scriptsize{$0$},\scriptsize{$\frac{\fd\mu}{2}$},\scriptsize{$\fd\mu$},\scriptsize{$\frd\mu$},\scriptsize{$\sd\mu$},\scriptsize{$d$}}
      ]
      \addplot [
          domain=-0:0.89,
          samples=100, 
          color=plotblue,
          name path=HD
      ]
      {\HD};
      \addplot [
          domain=0:1, 
          samples=100, 
          style=dashed,
          color=black,
          name path=lower,
          thick,
      ]
      {\FD + x*(\alpha-\FD/2)};
      \addplot [
          domain=0.89:1, 
          samples=100, 
          style=dashed,
          color=black,
          thick,
      ]
      {\HD};

      \addplot [
        domain=0:1,
        samples=100,
        style=dashed,
        color = black,
        name path=dtheta,
        thick
      ]
      {x};
      
      \addplot [
        domain=0:1,
        samples=100,
        color = black,
        thick,
      ]{x*(\HD-\FD) + \FD};
      \addplot[plotblue] fill between[of=lower and HD, soft clip={domain=0:0.45}];
      \addplot[plotblue] fill between[of=dtheta and HD, soft clip={domain=0.44:0.89}];
      \end{axis}
    \end{tikzpicture}
  \end{subfigure}
  \begin{subfigure}{.5\textwidth}
    \centering
    \begin{tikzpicture}[scale = 0.7]
      \def\HD{0.9}
      \def\alpha{0.6}
      \def\FD{0.7}
      \def\halfFD{0.35}
      \begin{axis}[
          axis lines = left,
          ymin = 0,
          ymax = 1.2,
          xmin = 0,
          width=9.1cm,
          xmax = 1.05,
          xlabel=\scriptsize{$\theta$},
          xtick = {1},
          ytick = {0,\halfFD,\FD, \alpha, \HD,1},
          xticklabels = {\scriptsize{$1$}},
          yticklabels = {\scriptsize{$0$},\scriptsize{$\frac{\fd\mu}{2}$},\scriptsize{$\fd\mu$},\scriptsize{$\frd\mu$},\scriptsize{$\sd\mu$},\scriptsize{$d$}}
      ]
      \addplot [
          domain=-0:0.8,
          samples=100, 
          color=plotblue,
          name path=HD
      ]
      {\HD};
      \addplot [
          domain=0:1, 
          samples=100, 
          style=dashed,
          color=black,
          name path=lower,
          thick,
      ]
      {\FD + x*(\alpha-\FD/2)};
      \addplot [
          domain=0.8:1, 
          samples=100, 
          style=dashed,
          color=black,
          thick,
      ]
      {\HD};

      \addplot [
        domain=0:1,
        samples=100,
        style=dashed,
        color = black,
        name path=dtheta,
        thick
      ]
      {x};
      
      \addplot [
        domain=0:1,
        samples=100,
        color = black,
        thick,
      ]{x*(\HD-\FD) + \FD};
      \addplot[plotblue] fill between[of=lower and HD, soft clip={domain=0:0.85}];
      \end{axis}
    \end{tikzpicture}
  \end{subfigure}%
  \begin{subfigure}{.5\textwidth}
    \centering
    \begin{tikzpicture}[scale = 0.7]
      \def\HD{0.9}
      \def\alpha{0.8}
      \def\FD{0.4}
      \begin{axis}[
          axis lines = left,
          ymin = 0,
          ymax = 1.2,
          xmin = 0,
          width=9.1cm,
          xmax = 1.05,
          xlabel=\scriptsize{$\theta$},
          xtick = {1},
          ytick = {0,\FD, \alpha, \HD,1},
          xticklabels = {\scriptsize{$1$}},
          yticklabels = {\scriptsize{$0$},\scriptsize{$\fd\mu$},\scriptsize{$\frd\mu$},\scriptsize{$\sd\mu$},\scriptsize{$d$}}
      ]
      \addplot [
          domain=-0:0.83,
          samples=100, 
          color=plotblue,
          name path=HD
      ]
      {\HD};
      \addplot [
          domain=0:0.83, 
          samples=100, 
          style=dashed,
          color=black,
          name path=lower,
          thick,
      ]
      {\FD + x*(\alpha-\FD/2)};
      \addplot [
          domain=0.83:1, 
          samples=100, 
          style=dashed,
          color=black,
          thick,
      ]
      {\HD};

      \addplot [
        domain=0:1,
        samples=100,
        style=dashed,
        color = black,
        name path=dtheta,
        thick
      ]
      {x};
      
      \addplot [
        domain=0:1,
        samples=100,
        color = black,
        thick,
      ]{x*(\HD-\FD) + \FD};
      \addplot[plotblue] fill between[of=lower and HD, soft clip={domain=0:0.83}];
      \end{axis}
    \end{tikzpicture}
  \end{subfigure}
  \caption{In order to improve the Stein--Tomas range for the restriction problem, we need the Fourier spectrum of $\mu$ to intersect the shaded region, i.e. for some $\theta\in[0,1]$ we need the point ($\theta, \fs \mu)$ to lie in the shaded region. Top left: when $\sd\mu\geq\frd\mu+\frac{\fd\mu}{2}$ and $\frd\mu>\fd\mu$. Top right: when $\sd\mu>\frd\mu + \frac{\fd\mu}{2}$ and $\frd\mu<\fd\mu$. Bottom left: when $\sd\mu<\frd\mu+\frac{\fd\mu}{2}$ and $\frd\mu < \fd\mu$. Bottom right: when $\sd\mu<\frd\mu+\frac{\fd\mu}{2}$  and $\frd\mu > \fd\mu$. The dashed lines are $\theta(\frd\mu - \tfrac{\fd\mu}{2}) + \fd\mu$ and $d\theta$, and the solid line is $\theta(\sd\mu - \fd\mu) + \fd\mu$, which by concavity is always a lower bound for the Fourier spectrum of $\mu$.}
  \label{fig:STimprovement}
  \end{figure}

\section{Application 2: A Stein--Tomas restriction theorem for the Sobolev dimension}

If $\fd\mu>0$, then there is an interval of $\theta$ for which $\fs \mu >d\theta$ (within which we can apply our estimate from  Theorems \ref{thm:mainthm} and \ref{thm:endpoint}). If  we   fix $\theta = 0$, then we (unsurprisingly) obtain the Stein--Tomas range. However, it is interesting to examine what happens at the other end, that is, Corollary \ref{cor:restrictionSpectrum}.  If we apply Corollary \ref{cor:restrictionSpectrum} and use only the trivial estimate from the Fourier spectrum coming from concavity, we obtain a restriction theorem in terms of the Sobolev and Fourier dimensions of $\mu$.  This result, despite being stated in terms of familiar notions of dimension, has been hidden from view up until now perhaps because it is obtained using the Fourier spectrum for a necessarily intermediate $\theta \in (0,1)$. 

While the Frostman dimension gives information regarding the scaling properties of a measure, the Sobolev dimension quantifies how rough the measure is, where roughness is measured in the Sobolev scale. For surface measures   on  manifolds such as the sphere or the cone, these two notions are equivalent. However, for fractal measures the Frostman  dimension will  often be strictly smaller than the Sobolev dimension.

Applying Theorem~\ref{thm:mainthm} with $\theta$ at the right end of the interval of allowable $\theta$ and using concavity of the Fourier spectrum, we get the following restriction theorem for the Sobolev dimension.  In fact, it follows by applying Corollary \ref{cor:restrictionSpectrum}, which we note because its proof is direct from the Fourier spectrum and does not involve the $L^2\to L^2$ estimate coming from the Frostman exponent. 
\begin{cor}\label{thm:restrictionSobolev}
  Let $\mu$ be a non-zero, finite, compactly supported, Borel measure on $\rd$ with $\sd\mu < d$. If
  \begin{equation*}
      q>4 + \frac{4(d-\sd\mu)}{\fd\mu},
  \end{equation*}
  then for all $f\in L^2(\mu)$,
  \begin{equation*}
      \|\widehat{f\mu}\|_{L^q(\rd)}\lesssim \|f\|_{L^2(\mu)}.
  \end{equation*} 
  
  Equivalently, if
  \begin{equation*}
      1\leq q' < 1 + \frac{\fd\mu}{4(d - \sd\mu + \fd\mu) - \fd\mu},
  \end{equation*}
  then for all $f\in L^{q'}(\rd)$,
  \begin{equation*}
      \|\widehat{f}\,\|_{L^2(\mu)}\lesssim \|f\|_{L^{q'}(\rd)}.
  \end{equation*}
\end{cor}
\begin{proof}
  By concavity of the Fourier spectrum for measures, $\fs\mu\geq \theta(\sd\mu - \fd\mu) + \fd\mu$. Therefore, choosing $\theta$ such that $\theta(\sd\mu - \fd\mu)+\fd\mu = d\theta$, i.e. $\theta = \frac{\fd\mu}{d-\sd\mu+\fd\mu}$,   Corollary \ref{cor:restrictionSpectrum} directly gives that the desired restriction estimate holds for 
  \begin{equation*}
  q> \frac{4}{\theta} = \frac{4(d-\sd\mu + \fd\mu)}{\fd\mu},
  \end{equation*}
as required.
\end{proof}

Similarly, by concavity (see \cite[Proof of Theorem~1.1]{Fra24}), we obtain the following corollary from Theorem~\ref{thm:endpoint}.
\begin{cor}\label{thm:restrictionSobolevEndpoint}
  Let $0<s<d$, $0<t<s$, and let $\mu$ be a non-zero, compactly supported, Borel measure on $\rd$, such that
  \begin{equation*}
      I_{s}(\mu)< \infty;
  \end{equation*}
  and
  \begin{equation*}
      \big| \widehat{\mu}(\xi) \big|\lesssim |\xi|^{-\frac{t}{2}}.
  \end{equation*}
  Then for all $f\in L^2(\mu)$,
  \begin{equation*}
      \|\widehat{f\mu}\|_{L^{4 + \frac{4(d-s)}{t}}(\rd)}\lesssim \|f\|_{L^2(\mu)}.
  \end{equation*}

  Equivalently, for all $f\in L^{1+\frac{t}{4(d - s + t) - t}}(\rd)$,
  \begin{equation*}
      \|\widehat{f}\,\|_{L^{2}(\mu)}\lesssim \|f\|_{L^{1+\frac{t}{4(d - s + t) - t}}(\rd)}.
  \end{equation*}
\end{cor}

The previous corollaries beat the Stein--Tomas range for measures $\mu$ with $\frd\mu + \fd\mu/2<\sd\mu$.  For fractal or multifractal measures this is quite typical, since the Fourier dimension is often small and the Frostman dimension is often strictly smaller than the Sobolev dimension (the Frostman dimension can even be as small as $\fd\mu/2$  \cite{LL24+}).  On the other hand,  if the Frostman and Sobolev dimensions agree, then the previous corollary never beats Stein--Tomas, and this is the case for surface measures.  That said,  Corollary \ref{cor:restrictionSpectrum} may beat Stein--Tomas, even for surface measures.

\section{Application 3: Restriction for the cone} \label{sec:cone}

In this section we exhibit our main theorem by applying it to an important example in restriction theory. Let $d \geq 3$ and define the truncated cone in $\rd$ by
\begin{equation*}
    C^{d-1} = \{ (t,|t|) : t\in\R^{d-1},\,|t|\leq 1 \}\subseteq\rd,
\end{equation*}
and let $\nu_{d-1}$ be the surface measure on the cone $C^{d-1}$. This is a well-known and important example in restriction theory.

\begin{conj}[cone restriction conjecture] \label{coneconj}
  The estimate $\|\widehat{f}\|_{L^{p'}(\nu_{d-1})}\lesssim_{p,q} \|f\|_{L^{q'}(\rd)}$ holds for all $f\in L^{q'}(\rd)$ if and only if
  \begin{equation*}
      \frac{d-2}{p'}\geq \frac{d}{q} \quad \text{and} \quad q>\frac{2(d-1)}{d-2},
  \end{equation*}
  and for $p=2$ if and only if
  \begin{equation*}
      q\geq  \frac{2d}{d-2}.
  \end{equation*}
\end{conj}

The general restriction conjecture for the cone has been proven only in dimensions $d = 3, 4, 5$ by Bacelo \cite{Bar85}, Wolff \cite{Wol01}, and Ou--Wang \cite{OW22}, respectively.   The case $p=2$ was established in unpublished work of Córdoba and Stein in all dimensions, see also \cite{Str77} where Strichartz generalised the result to quadratic surfaces.

In order to apply our results to the surface measure on the cone, we first derive the Fourier spectrum for this measure.  This result may be of independent interest.  
\begin{prop} \label{thm:coneSpectrum}
  Let $d \geq 3$ be an integer,  and $\nu_{d-1}$ be the surface measure on the cone.
  Then
  \[
  \fs \nu_{d-1} =  \min\Big\{ 2+ (d-1)\theta   ,  \ (d-2) + \theta  \Big\}
  \]
  for all $\theta \in [0,1]$.  This has a phase transition at $\theta = \frac{d-4}{d-2}$ whenever $d\geq 5$, but for $d=3,4$ it is affine.  Moreover, the Fourier dimension of $\nu_{d-1}$ is $\min\{2,d-2\}$  and the Sobolev (and Frostman) dimension of $\nu_{d-1}$  is $d-1$. 
  \end{prop}

  \begin{figure}[H]
    \begin{tikzpicture}[scale = 0.8]
      \begin{axis}[
          axis lines = left,
          xmin = 0,
          xmax = 1.05,
          ymin= 0,
          ymax = 5.2,
          ytick = {1,2,3,4,5},
          yticklabels = {$1$,$2$,$3$,$4$,$5$},
          xlabel=$\theta$,
          xtick = {0.2,0.4,0.6,0.8,1},
          xticklabels = {$0.2$,$0.4$,$0.6$,$0.8$,$1$},
      ]
      \foreach \d in {3,4,5,6} {
        \addplot [
          domain=0:1, 
          thick,
          samples=100, 
      ]
      {min(2 + (\d-1)*x, \d-2 + x)};
      }
      \end{axis}
  \end{tikzpicture}
  \caption{The Fourier spectrum of $\nu_{d-1}$ on the cone $C^{d-1}$ for $d = 3, \ldots, 6$; see Proposition~\ref{thm:coneSpectrum}.}
\end{figure}
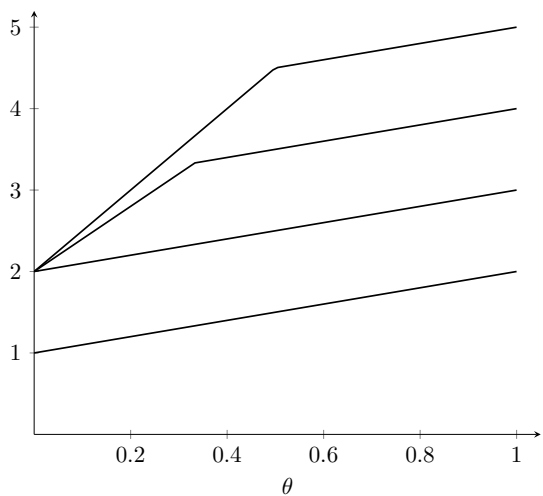

  We defer the proof of the theorem to the end of the section; see Subsection~\ref{coneproof}.
  
  From Theorem~\ref{thm:mainthm} we know that $\|\widehat{f\nu_{d-1}}\|_{L^q(\rd)}\lesssim \|f\|_{L^2(\nu_{d-1})}$ if
  \begin{equation*}
      q> 2 + 2\inf\limits_{\substack{\theta\in[0,1] \\ \fs\mu>d\theta}}\frac{(d-\alpha)(2-\theta)}{\fs\mu-\alpha \theta},
  \end{equation*}
  For $d=3,4$ the spectrum is affine, and we can do nothing more than recover Stein--Tomas (respectively, $q>6,4$), but this range is sharp in both of these cases.   For $d \geq 5$ we can do better than Stein--Tomas.   Here the infimum above is realised at the phase transition $\theta = \frac{d-4}{d-2}$ and, therefore, for any $f\in L^2(\nu_{d-1})$, the estimate $\|\widehat{f\nu_{d-1}}\|_{L^q(\rd)}\lesssim \|f\|_{L^2(\nu_{d-1})}$ holds if 
  \begin{equation*}
      q > \frac{3d - 4}{d-2}.
  \end{equation*}
The Stein--Tomas range for all $d \geq 5$ is $q>4$.  We also get a converse result from Theorem \ref{converse}.  For all $d \geq 3$, the Fourier spectrum crosses the diagonal $d\theta$ at $\theta = \frac{d-2}{d-1}$ and so Theorem \ref{converse} implies that the general restriction estimate \eqref{eq:restriction} fails for 
\[
q<\frac{2(d-1)}{d-2}.
\]
This outperforms the estimate $q<\frac{2d}{\sd \nu_{d-1}} = \frac{2d}{ d-1 }$ for all $d \geq 3$ and, moreover, recovers the second condition in the cone restriction conjecture \ref{coneconj}. 

  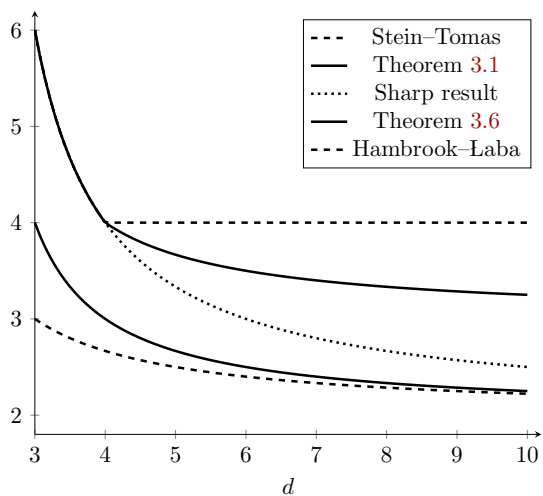
\begin{figure}[H]
    \begin{tikzpicture}[scale=0.8]
      \begin{axis}[
        axis lines = left,
        xmin = 3,
        xmax = 10.2,
        ymin= 1.8,
        ymax = 6.2,
        ytick = {2, 3,4,5,6},
        yticklabels = {$2$, $3$,$4$,$5$,$6$},
        xlabel=$d$,
        xtick = {3,4, 5, 6, 7, 8, 9, 10},
        xticklabels = {$3$,$4$,$5$,$6$,$7$,$8$,$9$,$10$},
    ]
   
    \addplot [ 
        domain=3:10, 
        thick,
        samples=100, 
        very thick,
        dashed
    ]
    {max(4,(2*x)/(x-2))};
    \addplot [ 
    domain=3:10, 
    very thick,
    samples=100,
    ]
    {max((2*x)/(x-2),(3*x - 4)/(x-2))};

    \addplot [ 
    domain=3:10, 
    very thick,
    samples=100, 
    dotted,
    ]
    {max(2*(x-1)/(x-2),(2*x)/(x-2))};

    \addplot [ 
    domain=3:10, 
    very thick,
    samples=100, 
    ]
    {2*(x-1)/(x-2)};

    \addplot[ 
        domain=3:10,
        very thick,
        samples=100,
        dashed,
    ]
    {2*x/(x-1)};

    \legend{Stein--Tomas, Theorem~\ref{thm:mainthm}, Sharp result,Theorem \ref{converse}, Hambrook--{\L}aba}
    \end{axis}
    \end{tikzpicture}
    \caption{Bounds for the range of $q$ for the restriction estimate \eqref{eq:extension} to hold for the cone in $\rd$. The dashed lines are the  Stein--Tomas upper bound and the Hambrook--{\L}aba lower bound, the dotted line is the sharp result, and the solid lines are our upper and lower bounds for the threshold.  These plots should be understood as only applying to integer points in the domain, but we included the full curve for aesthetic reasons.}\label{fig:restrictionCone}
  \end{figure}

  \subsection{Proof of Proposition~\ref{thm:coneSpectrum}}  \label{coneproof}

  Throughout the proof we   write $z=(x,y) \in \mathbb{R}^d$ with $x \in \mathbb{R}^{d-1}$ and $y \in \mathbb{R}$ and $(u,v) \in \mathbb{R}^d$ with $u \in \mathbb{R}^{d-1}$ and $v \in \mathbb{R}$. We also write $a\vee b = \max\{ a,b \}$ and $a\wedge b = \min\{ a,b \}$ for real numbers $a,b$. First observe that, up to normalisation constant,  $\nu_{d-1}$ disintegrates as
  \[
  \int f(u,v) \, d \nu_{d-1}(u,v) = \int_{0}^{1}   v^{d-2}\int_{S^{d-2}} f(uv,v) \, d \sigma_{d-2}(u) \, dv
  \]
  for $f \in L^1(\nu_{d-1})$ where $\sigma_{d-2}$ is the surface measure on the $(d-2)$-dimensional sphere $S^{d-2}$ in $\mathbb{R}^{d-1}$. Therefore,
  \begin{align*}
  \widehat{\nu_{d-1}} (z)  &= \int_{0}^{1}  v^{d-2} \int_{S^{d-2}} e^{-2\pi i (uv \cdot x + v \cdot y)}  \, d \sigma_{d-2}(u) \, dv  \\ 
  &= \int_{0}^{1}  v^{d-2} \widehat{ \sigma_{d-2}}(vx)   e^{-2\pi i  y  v}  \, dv \\  
  &= c(d-1) |x|^{\frac{3-d}{2}} \int_{0}^{1}  v^{\frac{d-1}{2}} J_{\frac{d-3}{2}}(2 \pi v |x|)  e^{-2\pi i  y v}  \, dv  \qquad \text{(by \cite[(3.41)]{Mat15})}
  \end{align*}
where $c(d-1)$ is a constant depending on the ambient dimension $d-1$ and $J_{\frac{d-3}{2}}$ is a Bessel function. We   use the above expression to derive estimates for $|\widehat{\nu_{d-1}} (z)| $ in different situations depending on $x,y$.

Recall the standard  asymptotic for   Bessel functions (e.g.~\cite[(3.37)]{Mat15})  which gives   
\begin{equation} \label{besselcos}
J_{m}(t) =     \frac{\sqrt{2}}{\sqrt{\pi t}} \cos\big(  t -\pi m/2- \pi/4\big) + R_m(t)
\end{equation}
where $R_m$ is a smooth error function satisfying
\[
|R_m(t) | \lesssim t^{-3/2}.
\]
Inserting this above gives
  \begin{align}
  \lvert \widehat{\nu_{d-1}} (z) \rvert &\lesssim  |x|^{\frac{2-d}{2}} \left\lvert  \int_{0}^{1}  v^{\frac{d-2}{2}} 
     \cos\big(2 \pi v |x| - \tfrac{\pi}{4}(d-2)\big)
   e^{-2\pi i  y v}  \, dv \right\rvert   \nonumber \\
& \hspace{3cm} + \ |x|^{\frac{3-d}{2}} \left\lvert  \int_{0}^{1}  v^{\frac{d-1}{2}} 
  R_{\frac{d-3}{2}} (2\pi v |x|)   e^{-2\pi i  y v}  \, dv \right\rvert.  \label{transform}
  \end{align}
We estimate these integrals separately, starting with the first one. Both estimates use   the method of stationary phase. Using Euler's identity,
  \begin{align}
  &\, 2\left\lvert  \int_{0}^{1}  v^{\frac{d-2}{2}} 
     \cos\big(2 \pi v |x| - \tfrac{\pi}{4}(d-2)\big)
   e^{-2\pi i  y v}  \, dv \right\rvert \nonumber\\ 
   &\leq \left\lvert  \int_{0}^{1}  v^{\frac{d-2}{2}} 
   e^{-2\pi i  y v+i \big(2 \pi v |x| - \tfrac{\pi}{4}(d-2)\big)}  \, dv \right\rvert  + \left\lvert  \int_{0}^{1}  v^{\frac{d-2}{2}} 
      e^{-2\pi i  y v -i \big(2 \pi v |x| - \tfrac{\pi}{4}(d-2)\big)}  \, dv \right\rvert. \label{integralstoest}
  \end{align}
  Both of these integrals can be expressed as
  \[
  \int_{0}^{1} \psi(v)   e^{i  |y| \phi(v)   }  \, dv
  \]
  for 
  \[
  \psi(v) = v^{\frac{d-2}{2}}  e^{\pm i\tfrac{\pi}{4}(d-2)}
  \]
  and
  \[
  \phi(v) =   -2\pi v \Big(\text{sgn}(y)  \pm  \frac{|x|}{|y|}\Big).
  \]
  Then
  \[
  \phi'(v) \equiv -2\pi \Big(\text{sgn}(y)  \pm  \frac{|x|}{|y|}\Big)
  \]
   and, by the stationary phase estimate given in \cite[Corollary~14.3]{Mat15},
  \[
  \left\lvert \int_{0}^{1} \psi(v)   e^{i  |y| \phi(v)   }  \, dv   \right\rvert \lesssim |y|^{-1} \left( |\psi(1) | + \int_0^1 |\psi'(v)| \, dv \right) \lesssim |y|^{-1}
  \]
  whenever $\frac{|x|}{|y|}$ is bounded away from 1.  This gives the following bound for the first term in \eqref{transform}:
  \begin{equation} \label{great1}
 |x|^{\frac{2-d}{2}} \left\lvert  \int_{0}^{1}  v^{\frac{d-2}{2}} 
     \cos\big(2 \pi v |x| - \tfrac{\pi}{4}(d-2)\big)
   e^{-2\pi i  y v}  \, dv \right\rvert    \lesssim  |x|^{\frac{2-d}{2}}|y|^{-1} 
  \end{equation}
  whenever $ \frac{|x|}{|y|}$ is bounded away from 1.

  Now we consider the second term from \eqref{transform}. By the stationary phase estimate given in \cite[Corollary~14.3]{Mat15},  we get
\begin{align*}  
\left\lvert\int_{0}^{1}  v^{\frac{d-1}{2}} 
  R_{\frac{d-3}{2}} (2\pi v |x|)   e^{-2\pi i  y v}  \, dv \right\rvert & \lesssim  |y|^{-1} \left( | R_{\frac{d-3}{2}} (2\pi |x|)| + \int_0^1 \left\lvert \frac{d}{dv}  v^{\frac{d-1}{2}}  R_{\frac{d-3}{2}} (2\pi v |x|) \right\rvert \, dv\right)   \\
&\lesssim  |y|^{-1} \left(|x|^{-3/2} + |x|^{-1/2} \right) 
\end{align*}
where we have used basic properties of Bessel functions (see \cite[(3.36), (3.37), and (3.38)]{Mat15}) to obtain the final estimate. This    gives the following bound for the second term in \eqref{transform}:
  \begin{equation} \label{great2}
|x|^{\frac{3-d}{2}} \left\lvert  \int_{0}^{1}  v^{\frac{d-1}{2}} 
  R_{\frac{d-3}{2}} (2\pi v |x|)   e^{-2\pi i  y v}  \, dv \right\rvert \lesssim |x|^{\frac{2-d}{2}}   |y|^{-1}.
  \end{equation}
Combining \eqref{great1},  \eqref{great2} and \eqref{transform} we get
 \begin{equation} \label{goodforbigy}
  \lvert \widehat{\nu_{d-1}} (z) \rvert \lesssim |x|^{\frac{2-d}{2}}   |y|^{-1}
\end{equation}
whenever $\frac{|x|}{|y|}$ is bounded away from 1. Note that the above calculation in fact gives
\begin{equation}\label{eq:withmins}
  \lvert \widehat{\nu_{d-1}} (z) \rvert \lesssim \big( |x|^{\frac{2-d}{2}} \wedge 1 \big)  \big(  |y|^{-1} \wedge 1 \big),
\end{equation}
whenever $\frac{|x|}{|y|}$ is bounded away from 1.  This aesthetically less pleasing but sharper estimate will be used in some cases below. 

For $\frac{|x|}{|y|}$ close to 1 we need a different approach. By \eqref{transform}, we always have the bound
  \begin{equation} \label{uniform}
  \lvert \widehat{\nu_{d-1}} (z) \rvert  \lesssim  |x|^{\frac{2-d}{2}}
  \end{equation}
  but can beat this estimate even when $ \frac{|x|}{|y|}$ is quite close to 1. For $  \frac{|x|}{|y|} \neq 1$ we re-express the integrals from \eqref{integralstoest} as
  \[
  \int_{0}^{1} \psi(v)   e^{i  |y  \pm  |x| | \phi(v)   }  \, dv
  \]
  for 
  \[
  \psi(v) = v^{\frac{d-2}{2}}  e^{\pm i\tfrac{\pi}{4}(d-2)}
  \]
  and
  \[
  \phi(v) =   -2\pi v \, \text{sgn}(y  \pm  |x|).
  \]
  Then \cite[Corollary~14.3]{Mat15} gives
  \[
  \int_{0}^{1} \psi(v)   e^{i  |y  \pm  |x| | \phi(v)   }  \, dv  \lesssim|y  \pm  |x| |^{-1}
  \]
  and, using  \eqref{transform},
  \begin{equation} \label{hard}
  \lvert \widehat{\nu_{d-1}} (z) \rvert  \lesssim  |x|^{\frac{2-d}{2}}|y  \pm  |x| |^{-1}
  \end{equation}
  which beats \eqref{uniform} whenever $|y  \pm  |x| | \geq 1$.   These estimates are enough for our purposes.

  From now on fix $\theta \in (0,1]$.  We begin with the lower bound.  Write
  \begin{align*}
  \J_{s,\theta}( \nu_{d-1})^{1/\theta}  &= \int_{\substack{z=(x,y)\in \mathbb{R}^{d-1} \times \mathbb{R}   : \\
  \frac{|x|}{|y|} \notin (1/2,2)}} \  \lvert \widehat{\nu_{d-1}} (z) \rvert ^{\frac{2}{\theta}}  |z|^{\frac{s}{\theta}-d} \, dz \\
  & +  \int_{\substack{z=(x,y)\in \mathbb{R}^{d-1} \times \mathbb{R}   : \\
  \frac{|x|}{|y|} \in (1/2,2)\text{ and }  
  |y  \pm  |x| | \geq 1}} \lvert \widehat{\nu_{d-1}} (z) \rvert^{\frac{2}{\theta}}  |z|^{\frac{s}{\theta}-d} \, dz \\
  & +  \int_{\substack{z=(x,y)\in \mathbb{R}^{d-1} \times \mathbb{R}   : \\
  \frac{|x|}{|y|} \in (1/2,2) \text{ and } |y  \pm  |x| | < 1}}  \lvert \widehat{\nu_{d-1}} (z) \rvert ^{\frac{2}{\theta}}  |z|^{\frac{s}{\theta}-d} \, dz \\ \\
  &=: I_1 + I_2+I_3.
  \end{align*}
We estimate these three integrals separately. Let $B_{k}(0,r)$ denote the    open $r$-ball  in $\mathbb{R}^{k}$ centred at the origin. First, using  \eqref{goodforbigy} and \eqref{eq:withmins},
  \begin{align*}
  I_1 &\lesssim \int_{\substack{  (x,y)\in \mathbb{R}^{d-1} \times \mathbb{R}   : \\
  |x| \geq |y|\geq1}}  \bigg(  |x|^{\frac{2-d}{2}}  |y|^{-1}    \bigg)^{\frac{2}{\theta}}  |x|^{\frac{s}{\theta}-d} \, dx \, dy \\
  & \qquad   \qquad   \qquad + \  \int_{\substack{(x,y)\in \mathbb{R}^{d-1} \times \mathbb{R}  : \\
  1\leq |x| \leq |y|}}    \bigg(  |x|^{\frac{2-d}{2}}  |y|^{-1}    \bigg)^{\frac{2}{\theta}}  |y|^{\frac{s}{\theta}-d} \, dx \, dy \\
  &= \int_{x\in \mathbb{R}^{d-1}\backslash B_{d-1}(0,1)}|x|^{\frac{s}{\theta}-d+\frac{2-d}{\theta}} \int_{\substack{y \in \mathbb{R}   : \\
    1\leq |y| \leq |x|}}    |y|^{-\frac{2}{\theta}} \, dy \, dx \\ 
  & \qquad   \qquad   \qquad + \  \int_{ y \in \mathbb{R}\backslash B_{1}(0,1) }  |y|^{\frac{s}{\theta}-d-\frac{2}{\theta}}  \int_{\substack{ x \in \mathbb{R}^{d-1} : \\
  1\leq |x| \leq |y|}}    |x|^{\frac{2-d}{\theta}}  \, dx \, dy\\
  &\lesssim \int_{x\in \mathbb{R}^{d-1}\backslash B_{d-1}(0,1)}|x|^{\frac{s}{\theta}-d+\frac{2-d}{\theta}}  \, dx  \  + \  \int_{ y \in \mathbb{R}\backslash B_{1}(0,1) }  |y|^{\frac{s}{\theta}-d-\frac{2}{\theta}+(d-1+\frac{2-d}{\theta}) \vee 0}  \, dy\\
  &< \infty
  \end{align*}
  provided
  \[
  \tfrac{s}{\theta}-d+\tfrac{2-d}{\theta} < 1-d
  \]
  and
  \[
  \tfrac{s}{\theta}-d-\tfrac{2}{\theta}+\big(d-1+\tfrac{2-d}{\theta}\big) \vee 0 < -1,
  \]
  that is, provided
  \[
  s <   \min\Big\{ 2+ (d-1)\theta   ,  \ (d-2) + \theta   \Big\}
  \]
  which is consistent with the  desired lower bound. 
  
  Turning to the second integral,   using  \eqref{hard}
  \begin{align*}
  I_2 &\lesssim \int_{\substack{(x,y)\in \mathbb{R}^{d-1} \times \mathbb{R}   : \\
  \frac{|x|}{|y|} \in (1/2,2)\text{ and }  
  |y  \pm  |x| | \geq 1}}  \Big( |x|^{\frac{2-d}{2}}|y  \pm  |x| |^{-1}\Big)^{\frac{2}{\theta}}  |x|^{\frac{s}{\theta}-d} \, dx \, dy \\
  &\lesssim \int_{\mathbb{R}^{d-1} \setminus B_{d-1}(0,1)}   |x|^{\frac{2+s-d}{\theta} -d} \int_{\substack{y \in \mathbb{R} :\\
 \frac{|x|}{|y|} \in (1/2,2)\text{ and }  
  |y  \pm  |x| | \geq 1}} |y  \pm  |x| |^{- \frac{2}{\theta}} \, dy  \, d x \\
  &\lesssim \int_{\mathbb{R}^{d-1} \setminus B_{d-1}(0,1)}   |x|^{\frac{2+s-d}{\theta} -d} \int_1^{|x|} \alpha^{- \frac{2}{\theta}} \, d\alpha  \, d x \\
  &\lesssim \int_{\mathbb{R}^{d-1} \setminus B_{d-1}(0,1)}   |x|^{\frac{2+s-d}{\theta} -d}   \, d x \\
  &<\infty
  \end{align*}
  provided
  \[
\frac{2+s-d}{\theta} -d < 1-d,
  \]
  that is, provided
  \[
  s <(d-2) + \theta 
  \]
  which is consistent with the  desired lower bound.
  
  Finally, for the third integral, using \eqref{uniform},
  \begin{align*}
  I_3 &\lesssim  \int_{x \in \mathbb{R}^{d-1} \setminus B_{d-1}(0,1)} \int_{\substack{y \in  \mathbb{R}   : \\
   |y  \pm  |x| | < 1}}  |x|^{\frac{2-d}{\theta}}  |x|^{\frac{s}{\theta}-d} \, dy \, dx \\ 
  &\lesssim  \int_{x \in \mathbb{R}^{d-1} \setminus B_{d-1}(0,1)}    |x|^{\frac{2+s-d}{\theta}-d}  dx \\ 
  &< \infty
  \end{align*}
  provided
  \[
  \tfrac{2+s-d}{\theta}-d < 1-d,
  \]
  that is, provided
  \[
  s <(d-2) + \theta 
  \]
  which is consistent with the  desired lower bound. Combining our estimates for $I_1$, $I_2$ and $I_3$, gives  the desired lower bound.

  We turn our attention to the upper bound. Fix $\eps \in (0,0.01)$  small.  Then
  \begin{align*}
  \J_{s,\theta}( \nu_{d-1})^{1/\theta} &=\int_{\mathbb{R}^{d-1}} \int_{\mathbb{R}}  
  |\widehat{\nu_{d-1}} (x,y) |^{\frac{2}{\theta}} 
   (|x| \vee |y|)^{\frac{s}{\theta}-d}\, dy \, dx \\
   &=\int_{\mathbb{R}^{d-1}} \int_{\mathbb{R}}  
  \left\lvert \int_{0}^{1}  v^{d-2} \widehat{ \sigma_{d-2}}(xv)   e^{-2\pi i  y \cdot v}  \, dv    \right\rvert^{\frac{2}{\theta}} 
   (|x| \vee |y|)^{\frac{s}{\theta}-d}\, dy \, dx \\
   &\geq   \int_{\mathbb{R}^{d-1}\setminus B_{d-1}(0,1)}  \int_{B_1(0,\eps)}  
  \left\lvert \int_{0}^{1}  v^{d-2} \widehat{ \sigma_{d-2}}(xv)   e^{-2\pi i  y \cdot v}  \, dv    \right\rvert^{\frac{2}{\theta}} 
   |x|^{\frac{s}{\theta}-d}\, dy \, dx \\
  & \gtrsim_\eps \int_{\mathbb{R}^{d-1} \setminus B_{d-1}(0,1)}    \left\lvert \int_{0}^{1}  v^{d-2} \widehat{ \sigma_{d-2}}(xv)    \, dv    \right\rvert^{\frac{2}{\theta}} |x|^{\frac{s}{\theta}-d}   \, dx \\
  & \approx \int_{\mathbb{R}^{d-1} \setminus B_{d-1}(0,1)}    |x|^{-\frac{d-2}{\theta}} |x|^{\frac{s}{\theta}-d}   \, dx \\
  &=\infty
  \end{align*}
  whenever $-\frac{d-2}{\theta} + \frac{s}{\theta}-d \geq -(d-1)$, proving $\fs \nu_{d-1} \leq  (d-2)+\theta$.  Similarly,
  \begin{align*}
  \J_{s,\theta}( \nu_{d-1})^{1/\theta} &\geq   \int_{\mathbb{R}\setminus B_{1}(0,1)}  \int_{B_{d-1}(0,\eps)}  
  \left\lvert \int_{0}^{1}  v^{d-2} \widehat{ \sigma_{d-2}}(xv)   e^{-2\pi i  y \cdot v}  \, dv    \right\rvert^{\frac{2}{\theta}} 
   |y|^{\frac{s}{\theta}-d} \, dx\, dy \\
  & \gtrsim_\eps   \int_{\mathbb{R}\setminus B_{1}(0,1)}   
  |y|^{-\frac{2}{\theta}} 
   |y|^{\frac{s}{\theta}-d}\, dy   \\
  &=\infty
  \end{align*}
  whenever $-\frac{2}{\theta} + \frac{s}{\theta}-d \geq -1$, proving $\fs \nu_{d-1} \leq  2+(d-1)\theta$.  This completes the proof.

\section{Application 4: Restriction for the moment curve} \label{sec:moment}

In this section we consider another important example in restriction theory and consider how our results compare with the Stein--Tomas range.  For $d \geq 2$, let the moment curve (or Veronese curve) be defined by
\[
V^d\{(t, t^2/2, \dots, t^d/{d!}): t \in [0,1]\} \subseteq \rd
\]
and let $\nu$ denote the  pushforward of Lebesgue measure onto the curve. Thus, $\nu$ is---up to a bi-Lipschitz density---the arclength  measure on $V^d$.  Despite being a 1-dimensional curve embedded in high ambient dimension,  one can still  get good estimates for Fourier decay and restriction due to the curvature.   It is known that the extension estimate 
\begin{equation} \label{eq:extensionmoment}
    \|\widehat{f\mu}\|_{L^{q}(\rd)}\lesssim \|f\|_{L^p(\mu)}
\end{equation}
holds if and only if 
 \begin{equation} \label{eq:momentconj}
  q \geq    \frac{d(d+1)p'}{2} \quad \text{and} \quad q>\frac{d^2+d+2}{2},
  \end{equation}
  and for $p=2$ if and only if  $q \geq d^2+d$; see \cite{drury}. We are most grateful to Jonathan Hickman for both suggesting the moment curve as an example to consider and completely solving it for us by pointing out that an explicit formula for the Fourier spectrum follows from estimates in \cite{bggist} for the averages
\[
G_\rho(R) := \bigg( \int_{S^{d-1}} |\widehat \nu(R \omega) |^\rho\, d\sigma_{d-1}(\omega) \bigg)^{1/\rho}.
\]
In particular, setting $\rho=\frac{2}{\theta}$ and using polar coordinates,
\[
\mathcal{J}_{s,\theta}(\nu)^{1/\theta} = \int_0^\infty R^{\frac{s}{\theta}-1} G_{\frac{2}{\theta}}(R)^{\frac{2}{\theta}} \, dR
\]
and applying \cite[Theorems 1.2--1.3, for $K=d$]{bggist}, we get the following. This example is noteworthy because we have not yet seen a natural example exhibiting multiple phase transitions. 
\begin{prop} \label{prop:moment}
Let $d \geq 2$ be an integer and $\nu$ be the arclength measure on the moment curve.  Then
\[
\fs \nu = \min_{k=2, \dots, d} \ \frac{2}{k} + \frac{k^2-k-2}{2k} \theta.
\]
In particular, $\fd \nu = 2/d$, the Fourier spectrum coincides with the Hausdorff dimension of the curve for $\theta \geq 1/2$ and the Fourier spectrum has $d-2$ phase transitions occurring at 
\[
\theta = \frac{4}{k^2-k+2} \qquad (k=3, \dots, d).
\]
Finally, $\fs \nu >d \theta $ holds for all 
\[
0\leq  \theta < \frac{4}{d^2+d+2}.
\]
\end{prop}
In fact \cite{bggist} treats much more general curves, in particular, allowing less curvature than the moment curve.  By carefully examining their results, one can derive the Fourier spectrum more generally, but we leave this to the interested reader. 

  \begin{figure}[H]
    \begin{tikzpicture}[scale = 0.8]
      \begin{axis}[
          axis lines = left,
          xmin = 0,
          xmax = 1.05,
          ymin= 0,
          ymax = 1.2,
          ytick = {0, 0.5,1},
          yticklabels = {$0$, $0.5$,$1$},
          xlabel=$\theta$,
          xtick = {0.2,0.4,0.6,0.8,1},
          xticklabels = {$0.2$,$0.4$,$0.6$,$0.8$,$1$},
      ]
     
        \addplot [
          domain=0:1, 
          thick,
          samples=100, 
      ]
      {min(2/2+(2^2-2-2)*x/(2*2), 2/3+(3^2-3-2)*x/(2*3),2/4+(4^2-4-2)*x/(2*4),2/5+(5^2-5-2)*x/(2*5),2/6+(6^2-6-2)*x/(2*6),2/7+(7^2-7-2)*x/(2*7),2/8+(8^2-8-2)*x/(2*8))};
     
      \end{axis}
  \end{tikzpicture}
  \caption{The Fourier spectrum of the arclength measure on the moment curve in $\mathbb{R}^8$; see Proposition~\ref{prop:moment}.  There are 6 phase transitions and the Fourier dimension is 1/4.}
\end{figure}
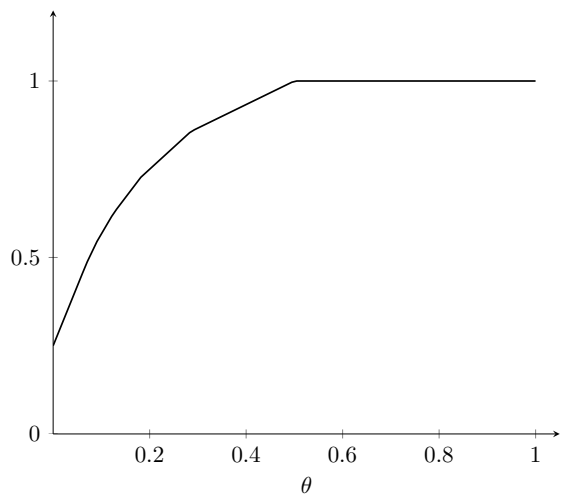

 Let us now examine  extension estimates coming from Proposition \ref{prop:moment}. For $p=2$, the Stein--Tomas range is
\[
q \geq 2d^2-2d+2
\]
and Theorem \ref{thm:mainthm} obtains the range
\[
q>d^2+d+2,
\]
which is the sharp bound plus 2. Moreover, for $d \geq 4$ the range in Theorem \ref{thm:mainthm} is optimised (uniquely) at $\theta =  \frac{4}{d^2+d+2}$.  On the other hand, Theorem \ref{converse} shows that the general extension estimate \eqref{eq:extensionmoment}  fails for 
\[
q<\frac{d^2+d+2}{2},
\]
whereas applying the Hambrook--{\L}aba bound only gives  failure for $q<d$. It is perhaps interesting to note that  the bound coming from Theorem \ref{converse} is sharp and provides the right-hand side estimate in \eqref{eq:momentconj}. This was also the case for the cone, recall Conjecture \ref{coneconj}, and the sphere, recall the discussion after Theorem \ref{converse}.

  \begin{figure}[H]
    \begin{tikzpicture}[scale=0.8]
      \begin{axis}[
        axis lines = left,
        xmin = 3,
        xmax = 8.2,
        ymin= 1.8,
        ymax = 114.2,
        ytick = {20, 40, 60, 80, 100},
        yticklabels = {$20$, $40$, $60$, $80$, $100$},
        xlabel=$d$,
        xtick = {3,4, 5, 6, 7, 8},
        xticklabels = {$3$,$4$,$5$,$6$,$7$,$8$},
legend style={at={(axis cs:3.5,75)},anchor=south west},
    ]
   
    \addplot [ 
        domain=3:8, 
        thick,
        samples=100, 
        very thick,
        dashed
    ]
    {2*x^2-2*x+2};
    \addplot [ 
    domain=3:8, 
    very thick,
    samples=100,
    ]
    {x^2+x+2};

    \addplot [ 
    domain=3:8, 
    very thick,
    samples=100, 
    dotted,
    ]
    {x^2+x};

    \addplot [ 
    domain=3:8, 
    very thick,
    samples=100, 
    ]
    {(x^2+x+2)/2};

    \addplot[ 
        domain=3:8,
        very thick,
        samples=100,
        dashed,
    ]
    {x};

    \legend{Stein--Tomas, Theorem~\ref{thm:mainthm}, Sharp result,Theorem \ref{converse}, Hambrook--{\L}aba}
    \end{axis}
    \end{tikzpicture}
    \caption{Bounds for the range of $q$ for the extension estimate \eqref{eq:extension} to hold for the moment curve in $\rd$. The dashed lines are the  Stein--Tomas upper bound and the Hambrook--{\L}aba  lower bound, the dotted line is the sharp result, and the solid lines are our upper and lower bounds for the threshold.  These plots should be understood as only applying to integer points in the domain, but we included the full curve for aesthetic reasons.}\label{fig:restrictionmoment}
  \end{figure}
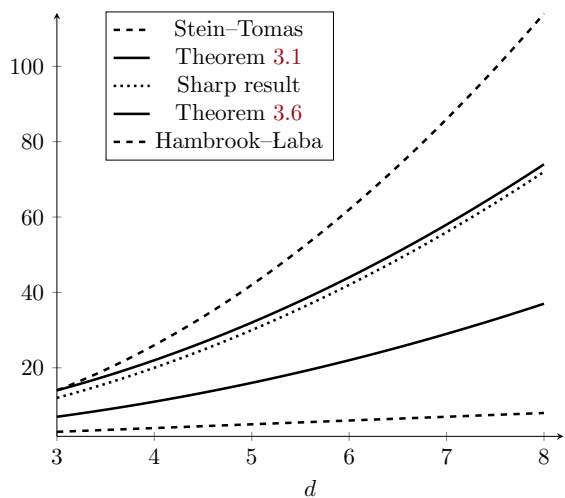

\section{Application 5: Restriction on fractals} \label{sec:fractals}

\subsection{Multifractal measures on the $1/3$-Cantor set} \label{sec:exampleCantor}

  We saw in Lemma~\ref{lma:decay} that it does not make sense to restrict the Fourier transform to measures that do not have Fourier decay. That said, examples of multifractal measures with positive Fourier dimension abound. Indeed, in \cite{Sol21} Solomyak proved that outside of a set of parameters of Hausdorff dimension zero, all self-similar measures supported on the line have positive Fourier dimension. However, finding explicit examples of such measures can be very difficult, and since we do not know of any, we will  add a small amount of `noise' to a concrete family of multifractal measures, which will result in a new family that will exhibit both multifractal behaviour and polynomial Fourier decay.

  Let $p>1/2$ and $\mu_{p}$ be the (self-similar) measure on the $1/3$ Cantor set associated to  probabilities $p$ and $(1-p)$. Then, by direct calculation or appealing to the extensive literature e.g. \cite{falconer},
  \begin{equation*}
  \fd \mu_p = 0;\quad  \frd\mu_{p} = \frac{\log p }{-\log 3};\quad \sd\mu_{p} = \frac{\log(p^2 +(1-p)^2)}{-\log3}.
  \end{equation*}
  Let $\varepsilon>0$ and $\nu_{\varepsilon}$ be a Salem measure of dimension $\varepsilon$.
  
  Consider the measure $ m = \mu_{p}*\nu_{\varepsilon}$. By \cite[Theorem~6.1]{Fra24}, for all $\theta\in[0,1]$, $\fs m \geq \fs\mu_p + \varepsilon$.

  If $\fd^{1/2}m >1/2$, we may fix $\theta = 1/2$ in Theorem~\ref{thm:mainthm}, which gives the range 
  \begin{equation}\label{eq:Cantorhalf}
      q > 2 + \frac{2(1-\alpha)\big(2-\tfrac{1}{2}\big)}{\fd^{1/2}m - \tfrac{\alpha}{2}} = 2 + \frac{6(1-\alpha)}{2\fd^{1/2}m - \alpha}
  \end{equation}
  for the restriction estimate to hold, where $\alpha = \frd m$.
  
  Since $\mu_{p}*\mu_{p}$  is a self-similar measure satisfying the open set condition with a system of $3$ maps and weights $p^2,~(1-p)^2$ and $2p(1-p)$, then the Fourier spectrum of $\mu_{p}$ at $\theta = 1/2$ is, by \cite[Lemma~6.2]{Fra24},
  \begin{equation*}
      \fd^{1/2}\mu_{p} = \frac{\sd(\mu_{p}*\mu_{p})}{2} = \frac{\log\big( p^4 + (2p(1-p))^2 + (1-p)^4 \big)}{-2\log3}.
  \end{equation*}
  Therefore,
  \begin{equation*}
      \fd^{1/2} m \geq \frac{\log\big( p^4 + (2p(1-p))^2 + (1-p)^4 \big)}{-2\log3}+ \varepsilon,
  \end{equation*}
  where we choose $\varepsilon$ small, but large enough so that $\fd^{1/2}m >1/2$.
  
  To bound $\alpha = \frd m$ from below let $\beta< \frd\mu$, $\delta<\varepsilon=\frd \nu_\eps$, $x\in\R$ and $r>0$.  Then
  \begin{align*}
    m\big( B(x,r) \big) &= \iint 1_{B(x,r)}(u+v)\,d\mu(u)\,d\nu_\eps(v)\\
      &= \int_{B(x,r)}\int_{B(x,r)-v}\,d\mu(u)\,d\nu_\eps(v)\\
      &= \int_{B(x,r)} \mu\big( B(x-v,r) \big)\,d\nu_\eps(v)\\
      &\lesssim r^{\beta + \delta},
  \end{align*}
 and  letting $\beta\to\frd\mu$ and $\delta\to\varepsilon$ shows that $\alpha\geq\frd\mu_{p} + \varepsilon$.
  
  Since $\fd^{1/2}m>1/2$, as a function of $\alpha$, $\frac{6(1-\alpha)}{2\fd^{1/2}m - \alpha}$ is decreasing. Replacing the obtained bounds on the right-hand side of \eqref{eq:Cantorhalf} we get that restriction is possible if 
  \begin{equation}\label{eq:exampleCFdO}
q>    2 + \frac{6(\log3 + \log p - \varepsilon\log 3)}{\log p + \varepsilon\log 3 - \log(p^4 + (2p(1-p))^2 + (1-p)^4)} .
  \end{equation}
  
  On the other hand,  Stein--Tomas gives the range
  \begin{equation}\label{eq:exampleST}
   q>  2 + \frac{4(\log 3 + \log p - \varepsilon\log3)}{\varepsilon\log3}.
  \end{equation}
  
For example, if $p = 0.6$, to have $\fd^{1/2} m >1/2$ we need $\varepsilon \geq 0.067$. Our results would still apply for smaller $\varepsilon$ (and always give a non-trivial range for restriction, see Subsection \ref{sec:khalil}) but we would need to use a smaller value of $\theta$.  The problem is then that we do not know the Fourier spectrum of $\mu_p$ explicitly  other than at $\theta=1/2$.  For this threshold value of $\varepsilon=0.067$, the range of \eqref{eq:exampleCFdO} is $q>7.99$ and the Stein--Tomas range  \eqref{eq:exampleST} is $q>29.95$.

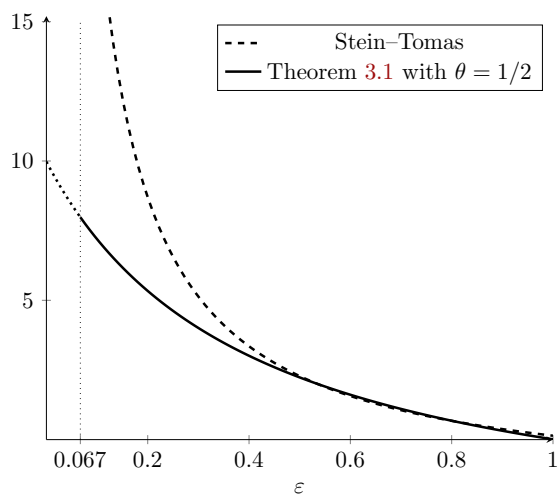
\begin{figure}[H]
  \begin{tikzpicture}[scale = 0.8]
    \def\p{0.6}
    \begin{axis}[
        axis lines = left,
        xmin = 0,
        xmax = 1,
        ymin= 0,
        ymax = 15.2,
        ytick = {5, 10, 15},
        yticklabels = {5,10,15},
        xlabel=$\varepsilon$,
        xtick = {0.067, 0.2,0.4,0.6,0.8,1},
        xticklabels = {$0.067$,$0.2$,$0.4$,$0.6$,$0.8$,$1$},
    ]

    \addplot [ 
        domain=0:1, 
        very thick,
        dashed,
        samples=100, 
    ]
    {2 + ( 4*( ln(3) + ln(\p) - x*ln(3) ) )/( x*ln(3) )};

    \addplot [ 
        domain=0.067:1, 
        very thick,
        samples=100, 
    ]
    {2 + ( 6*( ln(3) + ln(\p) - x*ln(3) ) )/( ln(\p) + x*ln(3) - ln((\p)^4 + (2*\p*(1-\p))^2 + (1-\p)^4) )};

    \legend{Stein--Tomas, Theorem~\ref{thm:mainthm} with $\theta = 1/2$}

    \addplot [ 
        domain=0:0.067, 
        very thick,
        dotted,
        samples=100, 
    ]
    {2 + ( 6*( ln(3) + ln(\p) - x*ln(3) ) )/( ln(\p) + x*ln(3) - ln((\p)^4 + (2*\p*(1-\p))^2 + (1-\p)^4) )};

    \addplot[
        domain=0:1,
        style=dotted,
      ] coordinates {(0.067,0)(0.067,15)};

    \end{axis}
\end{tikzpicture}
\caption{Lower bound on the range of $q$ for the restriction estimate \eqref{eq:extension} to hold for the measure $\mu_{0.6}*\nu_{\varepsilon}$, as a function of $\varepsilon$. Note that the bound obtained from Theorem~\ref{thm:mainthm} is only valid for values of $\varepsilon$ greater than $0.067$.}\label{fig:fractalExample}
\end{figure}

Curiously, if in the previous example we set $\varepsilon=0.7$, then from Figure~\ref{fig:fractalExample} or by direct calculation, we can see that the Stein--Tomas range \eqref{eq:exampleST} is better than the one given by Theorem~\ref{thm:mainthm} with $\theta = 1/2$, \eqref{eq:exampleCFdO}. This implies that for such values of $\varepsilon$, the improvement on the Stein--Tomas result will come from another value $\theta < 1/2$; see previous discussion and subsequent section.

\subsection{Measures on a large family of fractals} \label{sec:khalil}

Here we consider a variant on the family of examples from the previous section where we replace the concrete family of self-similar measures $\mu_p$ with a large family of measures satisfying a mild non-concentration condition (certainly the measures $\mu_p$ are in this family as well as many other examples). As in the previous example, these measures do not   have polynomial Fourier decay, and so in order to make sense of restriction we will convolve them with a Salem measure of small dimension. If we make this `noise'  small, we  get better restriction estimates than those  provided by Stein--Tomas'.

Let $\mu$ be a non-negative, finite, compactly supported, Borel measure on $\rd$ which satisfies:  $\fd\mu = 0$, $\sd \mu < d$, and that the upper right semi-derivative of $\theta \mapsto \fs\mu$ at $\theta =0$ is $d$.  In particular, since $\fs \mu \leq \fd \mu+ d \theta$ this derivative condition says that the Fourier spectrum has the largest possible growth at $\theta=0$. This seemingly artificial condition turns out to be satisfied very generally due to work of Khalil \cite{Kha23}. We say that a Borel measure $\mu'$ on $\rd$ is uniformly affinely non-concentrated if there exist $s>0$, $C\geq 1$ such that for every $\delta>0$, $x\in\rd$, $0<r\leq 1$ and  hyperplane $W\subseteq\rd$,
\begin{equation*}
    \mu'\big( W^{(\delta r)}\cap B(x,r) \big)\leq C \delta^s\mu\big( B(x,r) \big)
\end{equation*}
where $W^{(\delta r)}$ is the $\delta r$ neighbourhood of $W$.  In particular, many families of dynamically defined measures satisfy this condition. It turns out that if such a measure $\mu'$ has Fourier dimension zero, then it satisfies our derivative condition. Indeed, by \cite[Lemma~6.2]{Fra24} and \cite[Theorem~1.6]{Kha23},
\begin{equation*}
    \lim_{\theta\to0}\frac{\fs\mu'}{\theta} = d.
\end{equation*}
Let $\varepsilon>0$ be small and $\nu_{\varepsilon}$ be a Salem measure on $\rd$ of dimension $\varepsilon$. We consider the measure $m = \mu *\nu_{\varepsilon}$ and our goal is to establish a restriction estimate for $m$ with an optimal range as $\varepsilon \to 0$. Since we assume $\fd \mu = 0$, we cannot do any better in general than the estimate $\fd m \geq \varepsilon$ and therefore the Stein--Tomas range for restriction  \eqref{eq:restrictionp2} to hold is 
\begin{equation*}\label{eq:STKhalil}
    q>2 + \frac{4(d-\frd m)}{\varepsilon} =  \Omega(\varepsilon^{-1}).
\end{equation*}
Similarly, considering the Fourier spectrum, the best estimate we can get in general  is $\fs m \geq \fs \mu + \varepsilon$.  Let $\theta_0 \in (0,1)$ be chosen such that $\dim_\textup{F}^{\theta_0}  \mu + \varepsilon= d\theta_0$, noting that this choice is unique.  Then Corollary \ref{cor:restrictionSpectrum} gives that restriction  \eqref{eq:restrictionp2} holds for
\begin{equation*}
  q> \frac{4}{\theta_0} = o(\varepsilon^{-1}).
\end{equation*}
In particular, we obtain an asymptotically better range for restriction as $\varepsilon \to 0$. The fact that $\theta_0^{-1} = o(\varepsilon^{-1})$ follows since the    upper right semi-derivative of $\fs\mu$ at $\theta =0$ is $d$.   That is, for every $\eta>0$, there will exist $\theta>0$ such that   $(d - \eta)\theta<\fs\mu$ and thus $\theta_0>\varepsilon/\eta$ for $\varepsilon$ small enough.

In certain situations we can say more.  For example, if $\fs\mu$ is twice (right) continuously differentiable at $\theta =0$  (and still assuming that the first right derivative is $d$) then, by Taylor's theorem, 
\[
\fs \mu = d\theta - |O(\theta^2)|
\]
and so Corollary \ref{cor:restrictionSpectrum} gives that restriction  \eqref{eq:restrictionp2} holds for
\begin{equation*}
  q> \frac{4}{\theta_0} = O(\varepsilon^{-1/2}).
\end{equation*}

\section*{Acknowledgements}

We thank Jonathan Hickman, Chun-Kit Lai, and Bochen Liu  for helpful discussions and suggestions.


\begin{thebibliography}{BGGIST07}


	\bibitem[Bar85]{Bar85}B. Barcelo. On the Restriction of the Fourier Transform to a Conical Surface. {\em Trans. Amer. Math. Soc.}, \textbf{292}, pp. 321--333 (1985).
  
  \bibitem[BGGIST07]{bggist}L. Brandolini, G. Gigante, A. Greenleaf, A. Iosevich, A. Seeger, and G. Travaglini. Average decay estimates for Fourier transforms of measures supported on curves. \emph{J. Geom. Anal.}, \textbf{17}(1), pp. 15--40 (2007).

  \bibitem[Bou85]{Bou85}Bourgain, J. Estimations de certaines fonctions maximales. {\em C. R. Acad. Sci. Paris Sér. I Math.}. \textbf{301}, 499-502 (1985)

  \bibitem[BOS09]{BOS09}J.-G. Bak, D. M. Oberlin and A. Seeger. Restriction of Fourier Transforms to Curves and Related Oscillatory Integrals. {Amer. J. Math.}, \textbf{131}(2), 277--311 (2009)


  \bibitem[BS11]{BS11}J.-G. Bak and A. Seeger. Extensions of the Stein--Tomas Theorem. \emph{Math. Res. Lett.}, \textbf{18}(4), pp. 767--781 (2011).


  \bibitem[CFdO24+]{CFdO24}M. Carnovale, J. M. Fraser and A. E. de Orellana. Obtaining the Fourier spectrum via Fourier coefficients.  preprint:	\href{https://arxiv.org/abs/2403.12603}{arXiv:2403.12603} (2024).


  \bibitem[Che14]{Che14}X. Chen. A Fourier restriction theorem based on convolution powers. \emph{Proc. Amer. Math. Soc.}, \textbf{142}(11), pp. 3897--3901 (2014).

  \bibitem[Che16]{Che16}X. Chen. Sets of Salem type and sharpness of the $L^2$-Fourier restriction theorem. \emph{Trans. Amer. Math. Soc.}, \textbf{368}(3), pp. 1959--1977 (2016).
  
  \bibitem[CS72]{CS72}L. Carleson and P. Sj\"olin. Oscillatory integrals and multiplier problem for the disc. \emph{Studia Math.}, \textbf{44}, pp. 287--299 (1972).

  \bibitem[CS17]{CS17}X. Chen and A. Seeger. Convolution Powers of Salem Measures With Applications. \emph{Canad. J. Math.}, \textbf{69}(2), pp. 284--320 (2017).

  \bibitem[CSWW99]{CSWW99}A. Carbery, A. Seeger, S. Wainger and J. Wright. Classes of singular integral operators along variable lines. {\em J. Geom. Anal.}, \textbf{9}, pp. 583-605 (1999)

  \bibitem[Dem20]{Dem20}C. Demeter. \emph{Fourier Restriction, Decoupling, and Applications}. Cambridge University Press, (2020).

  \bibitem[Dru85]{drury} S. W. Drury. Restrictions of Fourier transforms to curves, \emph{Annales de l'institut Fourier}, {\bf 35}(1), pp. 117--123 (1985). 

\bibitem[Fal14]{falconer}K. J. Falconer. {\em Fractal Geometry: Mathematical Foundations and Applications}, John Wiley \& Sons, Hoboken, NJ, 3rd. ed., (2014).
  	
  \bibitem[Fef70]{Fef70}C. Fefferman. Inequalities for strongly singular convolution operators. \emph{Acta Math.}, \textbf{124}, pp. 9--36 (1970).

  \bibitem[Fra24]{Fra24}J. M. Fraser. The Fourier  spectrum and sumset type problems, \emph{Math. Ann.}, \textbf{390}, pp. 3891--3930 (2024).

  \bibitem[FdO24+]{FdO24}J. M. Fraser and A. E. de Orellana. A Fourier analytic approach to exceptional set estimates for orthogonal projections. \emph{Indiana U. Math. J.} (to appear), preprint:	\href{https://arxiv.org/abs/2404.11179}{arXiv:2404.11179} (2024).

  \bibitem[Gra14]{Gra14}L. Grafakos. Classical Fourier analysis. (New York, NY: Springer,2014)

  \bibitem[GXZ24]{GXZ24}Z. Gao, C. Xu and J. Zheng. Fourier Restriction Implies Maximal and Variational Fourier Restriction in Lorentz Space. \emph{Front. Math.}, \textbf{19}, pp. 609--628 (2024).
  
  \bibitem[Kha23+]{Kha23}O. Khalil. Exponential Mixing Via Additive Combinatorics. Preprint, available at: \href{https://arxiv.org/abs/2305.00527}{arXiv:2305.00527}, (2023).

  \bibitem[KT98]{KT98}M. Keel and T. Tao. Endpoint Strichartz estimates. \emph{Amer. J. Math.}, \textbf{120}(5), pp. 955--980 (1998).

  \bibitem[H{\L}13]{HL13}K. Hambrook and I. {\L}aba. On the Sharpness of Mockenhaupt’s Restriction Theorem. \emph{Geom. Funct. Anal.}, \textbf{23}, pp. 1262--1277 (2013).

  \bibitem[H{\L}16]{HL16}K. Hambrook and I. {\L}aba. On the Sharpness of Mockenhaupt--Mitsis--Bak--Seeger Restriction Theorem in Higher Dimensions. \emph{Bull. Lond. Math. Soc.}, \textbf{48}, pp. 757--770 (2016).

  \bibitem[LL24+]{LL24+}L. Li and B. Liu. Dimension of Diophantine approximation and applications. preprint:	\href{https://arxiv.org/abs/2409.12826}{arXiv:2409.12826} (2024).

  \bibitem[LP22]{LP22}Y. Liang and M. Pramanik. Fourier dimension and avoidance of linear patterns. \emph{Adv. Math.}, \textbf{399}(108252), (2022).

  \bibitem[{\L}W18]{LW18}I. {\L}aba and H. Wang. Decoupling and Near-optimal Restriction Estimates for Cantor Sets. \emph{Int. Math. Res. Not. IMRN}, \textbf{2018}(9), pp. 2944--2966 (2018).

  \bibitem[Mat15]{Mat15}P. Mattila. \emph{Fourier analysis and Hausdorff dimension}. Cambridge Studies in Advanced Mathematics, \textbf{150}, Cambridge, (2015).

  \bibitem[Mit02]{Mit02}T. Mitsis. A Stein–Tomas restriction theorem for general measures. {\em Publ. Math. Debrecen}, \textbf{60}, pp. 89--99 (2002).

  \bibitem[Moc00]{Moc00}G. Mockenhaupt. Salem sets and restriction properties of Fourier transforms. {\em Geom. Funct. Anal.}, \textbf{10}, pp. 1579--1587 (2000). 

  \bibitem[OeS24]{OeS24}D. Oliveira e Silva. Correction to: The endpoint Stein--Tomas inequality: old and new. \emph{S\~ao Paulo J. Math. Sci.} (2024).

  \bibitem[OW22]{OW22}Y. Ou and H. Wang. A cone restriction estimate using polynomial partitioning. \emph{J. Eur. Math. Soc.}, \textbf{24}, pp. 3557--3595 (2022).

  \bibitem[Sal51]{Sal51}K. Salem. On singular monotonic functions whose spectrum has a given Hausdorff dimension. \emph{Ark. Mat.}, \textbf{1}, pp. 353--365 (1951).

  \bibitem[Sol21]{Sol21}B. Solomyak. Fourier decay of self-similar measures. {\em Proc. Amer. Math. Soc.}, \textbf{149}, pp. 3277--3291 (2021).

  \bibitem[SS18]{SS18}P. Shmerkin and V. Suomala.  \emph{Spatially independent martingales, intersections, and applications}. Memoirs of the American Mathematical Society. \textbf{251}(1195), Cambridge, (2018).

  \bibitem[Ste93]{Ste93}E. M. Stein. \textit{Harmonic analysis: real-variable methods, orthogonality, and oscillatory integrals}. Princeton University Press, Princeton, NJ. With the assistance of T. S. Murphy, (1993).

  \bibitem[Str77]{Str77}R. Strichartz. Restrictions of Fourier transforms to quadratic surfaces and decay of solutions of wave equations. {\em Duke Math. J.}, \textbf{44}, pp. 477--478 (1977).

  \bibitem[SW71]{SW71}E. Stein and G. Weiss. \textit{Introduction to Fourier Analysis on Euclidean Spaces}. Princeton University Press, (1971).

  \bibitem[Tom75]{Tom75}P. A. Tomas. A restriction theorem for the Fourier transform. {\em Bull. Amer. Math. Soc.}, \textbf{81}, pp. 705--714 (1975).

  \bibitem[Wol01]{Wol01}T. Wolff. A Sharp Bilinear Cone Restriction Estimate. {\em Ann Of Math.}, \textbf{153}, pp. 661--698 (2001).
	
\end{thebibliography}
\end{document}